\documentclass[12pt]{amsart}
\usepackage{stablestyle,afterpage,pdflscape,graphicx}
\usepackage{color}
\usepackage{microtype}
\usepackage[margin=1in]{geometry}

\usepackage{mathtools}
\author{Espen Auseth Nielsen}
\title[Strictifying homotopy coherent actions on Hochschild complexes]{Strictifying homotopy coherent actions \\ on Hochschild complexes}

\newcommand{\Par}{\text{Par}}
\newcommand{\sCh}{\mathsf{sCh}}

\begin{document}

\newpage
\pagenumbering{arabic}

\begin{abstract} 
\noindent If $P$ is a dg-operad acting on a dg-algebra $A$ via algebra homomorphisms, then $P$ acts on the Hochschild complex of $A$. In the more general case when $P$ is a dg-prop, we show that $P$ still acts on the Hochschild complex, but only up to coherent homotopy. We moreover give a functorial dg-replacement of $P$ that strictifies the action. As an application, we obtain an explicit strictification of the homotopy coherent commutative Hopf algebra structure on the Hochschild complex of a commutative Hopf algebra.
\end{abstract}


\maketitle

\section{Introduction}

A \emph{dg-prop} \cite[Section 24]{maclanecatalg} is a symmetric monoidal dg-category $P$ whose monoid of objects is isomorphic to $(\N,+)$. If $P$ is a dg-prop and $\C$ a symmetric monoidal dg-category, then a \emph{$P$-algebra} in $\C$ is a symmetric monoidal functor $P\rightarrow \C$. A \emph{morphism} of dg-props is a symmetric monoidal dg-functor which induces an isomorphism of object monoids, and such a morphism is called a \emph{quasi-equivalence} if it induces quasi-isomorphisms on Hom-complexes. Let $\mathsf{dgprop}$ be the subcategory of $\dgCat$ generated by the dg-props and morphisms of dg-props. Examples of dg-props arise from dg-operads $\mathcal{O}$ (see e.g. \cite[Example 60]{markl}) by the formula
$$P(n,m)=\bigoplus_{n_1+...+n_m=n} \mathcal{O}(n_1)\otimes ... \otimes \mathcal{O}(n_m)$$
and composition defined using the composition product of $\mathcal{O}$. With this definition, an algebra over an operad is precisely an algebra over the dg-prop it generates (see e.g. \cite[p.10]{fiorenza}. In the following, we will therefore not distinguish between a dg-operad and the dg-prop it generates.

\bigskip The Hochschild complex is a functor from $\A_\infty$-algebras to chain complexes. If $P$ is a prop equipped with a morphism of props $\A_{\infty}\rightarrow P$, the Hochschild complex restricts to a functor from $P$-algebras. It is an open problem to compute the operations on the Hochschild complex of algebras over such props $P$. Partial results have been obtained in many cases, see e.g. \cite{wahl16,wahl12}. One such case is the following.

\bigskip If $P$ is a dg-operad (or a dg-prop) the Boardman-Vogt tensor product $\Ass\otimes P$ (see \cite[Section II.3]{bv}) is the dg-prop characterized by the equivalence
$$\Fun^{\otimes}(\Ass\otimes P, \C)\simeq \Fun^{\otimes}(P,\Alg(\C))$$
for any symmetric monoidal dg-category $\C$. Evaluating the right hand side at $1\in P$, we obtain a functor from $\Ass\otimes P$-algebras to $\Ass$-algebras. The Hochschild complex of an $\Ass\otimes P$-algebra is by definition the Hochschild complex of the associated $\Ass$-algebra. The Hochschild complex functor is lax monoidal. The structure morphisms may be used to prove that for a dg-operad $P$ and an algebra $A$ over the tensor product ${\Ass\otimes P}$, the Hochschild complex of $A$ admits a $P$-algebra structure (see \cite{brun07} and \cite[Section 6.9]{wahl16}). On the other hand, this fails if $P$ is a more general dg-prop. This is due to the failure of the Dold-Kan equivalence to be a symmetric monoidal equivalence. It is however true up to coherent homotopy, as the Dold-Kan equivalence is an $\mathbb{E}_{\infty}$-monoidal equivalence \cite[Section 5]{richter03}. In this paper, we give an explicit functorial strictification of the natural homotopy coherent $P$-algebra structure on the Hochschild complex of a $(P\otimes\Ass)$-algebra. Formally this is encoded in the following result.

\bigskip\noindent \textbf{Theorem A.} 
\textit{Let $\kappa$ be an inaccessible cardinal and let $k$ be a commutative ring with cardinality less than $\kappa$. Let $\Ch_k$ be the category of chain complexes over $k$ with cardinality less than $\kappa$ and let $\mathsf{dgprop}$ be the category of dg-props over $k$. There is a functor
$$\tilde{(-)}\colon  \mathsf{dgprop}\rightarrow \mathsf{dgprop}$$
equipped with a natural quasi-equivalence $\tilde{(-)}\rightarrow \id$ and a natural transformation
$$\tilde{\alpha}\colon \Fun^{\otimes}(\Ass\otimes -,\Ch_k)\rightarrow \Fun^{\otimes}(\tilde{(-)},\Ch_k)$$
of functors $\mathsf{dgprop}^{op}\rightarrow \Cat$ such that for a dg-prop $P$ and an $\Ass\otimes P$-algebra
$$\Phi\colon \Ass\otimes P\rightarrow \Ch_k$$
the value $\tilde{\alpha}_P(\Phi)(1)$ is equal to the Hochschild complex of $\Phi(1)$.}

\bigskip We use explicit generators and relations to construct the functor $\tilde{(-)}$, fattening the input prop with the structure maps of the Dold-Kan equivalence. The functor $\tilde{(-)}$ also admits the structure of a non-unital monad.

\bigskip\noindent \textbf{Example.} (Example \ref{example-chopf}) If $\Phi\colon \CHopf\rightarrow \dgAlg_k$ is a commutative Hopf algebra over any ring $k$, then
$$\tilde{\alpha}_{\CHopf}(\Phi)\colon \tilde{\CHopf}\rightarrow \Ch_k$$
gives an explicit strict model for the coherent commutative Hopf algebra structure of the Hochschild complex of $\Phi(1)$.

\bigskip Given a dg-prop $P$, one may ask whether $\tilde{P}$ is cofibrant in a model structure on dg-props. In \cite{fresse08}, Fresse constructs a model structure on the category of props over a field of characteristic zero, and a semi-model structure on certain sub-families of props in positive characteristic. However, for example Hopf algebras in positive characteristic cannot be treated in his framework. Additionally, in characteristic zero, our replacement $\tilde{P}$ will not be cofibrant. For example, our replacement of the commutative prop still has a strictly commutative multiplication. 

\bigskip\noindent\textbf{Further Questions.} Theorem A displays $\tilde{P}$ as a sub-prop of the prop of natural operations on the Hochschild complex. On the other hand, it leaves open the interaction of the $P$-action with Connes' $B$-operator. The determination of the total prop of natural operations on $C(A)$ is still an interesting open problem with a view toward operations on cyclic homology.

\bigskip The structure of the paper is as follows. In Section 2 we define the Dold-Kan structure maps and their action on Hochschild complexes, and establish necessary properties. In Section 3 we define the fattening functor $\tilde{(-)}$ for dg-props and prove the main theorem.

\begin{samepage}
\bigskip
\begin{center}
\textbf{Acknowledgements}
\end{center}
I am very grateful to my advisor Nathalie Wahl for helpful discussions, comments and proofreading, to David Sprehn for proofreading, and to Tobias Barthel for helpful comments. I am thankful to Marcel B\" okstedt, S{\o}ren Galatius, and Birgit Richter for helpful comments. The author was supported by the Danish National Research Foundation through the Centre for Symmetry and Deformation (DNRF92).
\end{samepage}

\section{The cyclic bar construction and the Dold-Kan equivalence}

In this section we will build a dg-category $\tilde{N}^{\Sigma}$ from the structure maps of the Dold-Kan correspondence and establish the action of $\tilde{N}^{\Sigma}$ on Hochschild complexes of dg-algebras. This dg-category is a key ingredient for the fattening functor we will construct in Section 3. 

\begin{convention}
Throughout, we fix a commutative base ring $k$. All algebras are assumed to be algebras over $k$. We employ the Kozul sign convention for chain complexes. In particular, our convention for bicomplexes are that the differentials anti-commute. Furthermore, throughout the paper we fix an inaccessible cardinal $\kappa$. All abelian groups (in particular, all simplicial modules and chain complexes over $k$) are assumed to have cardinality less than $\kappa$.
\end{convention}

We begin by recalling some basic notions from homological algebra.

\bigskip We will work with the categories $\sMod_k$ of simplicial $k$-modules and $\Ch_k$ of non-negatively graded chain complexes over $k$ with $k$-linear chain maps, where we consider $\Ch_k$ as a category enriched in itself. $\sMod_k$ is a symmetric monoidal simplicial category with tensor product given by the degreewise tensor product of $k$-modules. We denote this tensor product by $\hat{\otimes}$. Similarly, $\Ch_k$ is a symmetric monoidal category with tensor product denoted by $\otimes$ and given by $(A\otimes B)_*=\oplus_{p+q=*} A_p\otimes B_q$ and differential $d_{A\otimes B}(a\otimes b) = d_A(a)\otimes b + (-1)^{|a|} a\otimes d_B(b)$. The category of monoids in $\sMod_k$ is denoted by $\sAlg_k$, and is a symmetric monoidal category with the levelwise tensor product.

\bigskip For a simplicial chain complex $A=A_{*,\bullet}$ over $k$, call $*$ the differential degree and $\bullet$ the simplicial degree. Write $d^{a,b}_A\colon A_{a,b}\rightarrow A_{a-1,b}$ for the differential and $d_i^{a,b}\colon A_{a,b}\rightarrow A_{a,b-1}$ for the simplicial face maps. Write $\sCh_k$ for the category of simplicial chain complexes over $k$.

\bigskip The Dold-Kan equivalence
$$N\colon \sMod_k \rightleftarrows \Ch_k :  \Gamma$$
gives an equivalence of categories between simplicial $k$-modules and connective chain complexes over $k$. The functor $N\colon \sMod_k\rightarrow \Ch_k$, is called the \emph{normalized Moore complex} functor, and takes a simplicial $k$-module $M_{\bullet}$ to the chain complex $NM_*$ with $NM_p=M_p/sM_{p-1}$, the quotient of $M_p$ by the degenerate simplices, and $d\colon NM_p\rightarrow NM_{p-1}$ given by the alternating sum $d=\sum_{i=0}^p (-1)^{i}d_i$. The inverse functor $\Gamma\colon \Ch_k\rightarrow \sMod_k$ is called the Dold-Kan construction. We can also apply $N$ degreewise to a simplicial chain complex as follows:

\begin{defn}
\begin{enumerate}

\item For $A\in \sCh_k$, the \emph{bicomplex associated to $A_{*,\bullet}$} is denoted by $N_\epsilon(A_{*,\bullet})_*$ and is obtained by applying the Moore complex functor levelwise and shifting the differentials by the differential degree of $A$. Writing this out, we have $N_\epsilon(A_{a,\bullet})_b=A_{a,b}/sA_{a,b-1}$, the horizontal differential is $d_h=d_A$, and the vertical differential is
$$d_v^{a,b}=(-1)^{a} \sum_{i=0}^b (-1)^i d_i^{a,b}$$
We write
$$N_{\delta}(A) \colon = \text{Tot} \left(N_\epsilon(A_{*,\bullet})_*\right).$$

\item Let $A$ and $B$ be simplicial chain complexes over $k$ and denote by $A\hat{\otimes}B$ the simplicial chain complex which in simplicial degree $p$ is given by $A_{*,p}\otimes B_{*,p}$. The differential of $A\hat{\otimes}B$ is given by
$$d_{A\hat{\otimes}B}^{n,p}(a\otimes b)=d_A^{|a|,p}a\otimes b + (-1)^{|a|+p} a\otimes d_B^{|b|,p}b.$$
\end{enumerate}
\end{defn}

\begin{defn} \label{defn-cyclic-bar-construction} \cite[Section 5.3.2]{loday92} 
\begin{enumerate}
\item The \emph{cyclic bar construction} is the functor $B^{cy}\colon \dgAlg_k\rightarrow \sCh_k$ given in simplicial degree $p$ by $B^{cy}_p(A)=A^{\otimes p+1}$. The face maps $d_i\colon B^{cy}_p(A)\rightarrow B^{cy}_{p-1}(A)$ are given by
$$d_i\colon a_0\otimes ... \otimes a_p\mapsto \left\{ \begin{array}{ll} a_0\otimes ... \otimes a_i a_{i+1} \otimes ... \otimes a_p & , i=0,...,p-1 \\ (-1)^{|a_p|(|a_0|+...+|a_{p-1}|)} a_p a_0\otimes a_1\otimes ... \otimes a_{p-1} & , i=p \end{array} \right.$$
and the degeneracies $s_i\colon B^{cy}_p(A)\rightarrow B^{cy}_{p+1}(A)$ are given by
$$s_i\colon  a_0\otimes ... \otimes a_p\mapsto a_0\otimes ... \otimes a_i \otimes 1 \otimes a_{i+1}\otimes ... \otimes a_p$$
making $B^{cy}_{\bullet}(A)$ into a simplicial chain complex, called the \emph{cyclic bar construction} of $A$.
\item For a dg-algebra $A$, the complex
$$C(A)\colon = N_{\delta} B^{cy}_*(A)$$
is called the \emph{Hochschild complex} of $A$.
\end{enumerate}
\end{defn}

Note: The following lemma is doubtlessly well known, but the author was unable to find a reference which proves the result, so we give a proof here.

\begin{lem} \label{lem-bcyc-sym-mon}
The cyclic bar construction
$$B^{cy}\colon \dgAlg_k\rightarrow \sCh_k$$
is symmetric monoidal.
\end{lem}
\begin{proof}
Let $A$ and $B$ be dg-algebras over $k$. Define the natural transformation
$B^{cy}(A)\hat{\otimes} B^{cy}(B)\rightarrow B^{cy}(A\otimes B)$ is given in simplicial degree $p$ by permuting tensor factors:
$$A^{\otimes p+1}\otimes B^{\otimes p+1}\xrightarrow{\sigma} (A\otimes B)^{\otimes p+1}$$
$$a_0\otimes ... \otimes a_p \otimes b_0\otimes ... \otimes b_p \mapsto (-1)^{\sgn(a,b, \sigma)} a_0\otimes b_0 \otimes ... \otimes a_p\otimes b_p$$
where $\sgn(a,b, \sigma)\in \Z/2$ is the sign of $\sigma$ weighted by the elements $a_i,b_j$, which can be computed as
$$\sgn(a,b, \sigma) \equiv \sum_{i=0}^{p-1} |b_i|\left( \sum_{j=i+1}^{p} |a_j| \right) \quad (\text{mod 2})$$
To check that this defines a chain map in simplicial degree $p$, we must verify that there are no sign issues. It is sufficient to consider each summand of the differential separately. For the differential acting on $a_k$ for $1\leq k \leq p+1$, the sign we get by permuting first (i.e. the sign associated to $d^{A\otimes B} \circ \sigma$) is
$$\sgn(a,b, \sigma)+ \left(\sum_{i=0}^{k-1} |a_i|+|b_i|\right)$$
where the second term comes from the placement of $a_k$ after permuting.
The sign we get by permuting second is
$$\left(\sum_{i=0}^{k-1} |a_i|\right) + \sgn(a,b, \sigma) + \sum_{j=0}^{k-1} |b_j|$$
where the third term is the correction to $\sgn_{a,b}\sigma$ when the degree of $a_k$ is decreased by one. We see that the two are equal. For the differential acting on $b_k$, the sign we get by permuting first is
$$\sgn(a,b, \sigma) + \left(\sum_{i=0}^{k-1} |a_i|+|b_i|\right) + |a_k|$$
and the sign we get by permuting second is
$$\left(\sum_{i=0}^{p} |a_i|\right) + \left(\sum_{j=0}^{k-1} |b_j|\right) + \sgn(a,b, \sigma) + \sum_{j=k+1}^{p} |a_j|$$
where the fourth term is the correction to $\sgn_{a,b}\sigma$ when the degree of $b_k$ is decreased by one. Again we see that the two expressions are equal mod 2, hence we have a chain map.

We now verify that $\sigma$ is a symmetric monoidal transformation. Let $\tau$ be the symmetric monoidal twist map of $\Ch_k$, given by $A\otimes B \rightarrow B\otimes A$, $a\otimes b \mapsto (-1)^{|a||b|} b\otimes a$. We check that $\sigma\circ \tau_{p+1,p+1} = (\tau^{\otimes p+1})\circ \sigma$. The left hand side has sign
$$\sgn_L=\left(\sum_{i=0}^p |a_i|\right) \left( \sum_{j=0}^p |b_j|\right) + \sum_{i=0}^{p-1} |a_i|\left( \sum_{j=i+1}^{p} |b_j| \right) = \sum_{i=0}^{p} |a_i| \left( \sum_{j=0}^i |b_j| \right)$$
and the right hand side has sign
$$\sgn_R=\sum_{i=0}^{p-1} |b_i|\left( \sum_{j=i+1}^{p} |a_j| \right) + \left(\sum_{i=0}^{p} |a_i||b_i|\right) = \sum_{i=0}^p |b_i| \left( \sum_{j=i}^p |a_j| \right)$$
Rearranging the order of summation shows that $\sgn_L=\sgn_R$, and we conclude that $B^{cy}$ is symmetric monoidal as claimed.
\end{proof}

\bigskip Before discussing the monoidality properties of the cyclic bar construction and the Hochschild complex, we recall some of the monoidality properties of the Dold-Kan equivalence.

\begin{defn} \cite[5.3]{em53}
Let $A$ and $B$ be simplicial $k$-modules. The {\em shuffle map} (also called the {\em Eilenberg-Zilber map}):
$$\sh_{A,B}\colon N(A)\otimes N(B)\rightarrow N(A\hat\otimes B)$$
is defined on elementary tensors $a\otimes b\in A_{p}\otimes B_{q}$ as
$$\sh_{A,B}(a\otimes b)=\sum_{\sigma\in \Sigma_{(p,q)}} \sgn(\sigma) s_{\sigma(p+q)}...s_{\sigma(p+1)}a \otimes s_{\sigma(p)}...s_{\sigma(1)}b$$
When there is no risk of confusion, we will omit $A$ and $B$ from the notation and simply write $\sh$ for the shuffle map.
\end{defn}

\begin{defn} \cite[2.9]{em54}
Let $A$ and $B$ be simplicial $k$-modules. The {\em Alexander-Whitney map}
$$AW_{A,B}\colon  N(A\hat\otimes B)\rightarrow N(A)\otimes N(B)$$
is defined on elementary tensors $a\otimes b\in A_{n}\otimes B_{n}$ as
$$AW_{A,B}\colon (a\otimes b)\mapsto \sum_{i=0}^n d_{i+1}... d_{n-1}d_{n} a \otimes (d_0)^{i} b$$
As with the shuffle map, we omit $A,B$ from the notation $AW_{A,B}$ when there is no risk of confusion.
\end{defn}

\begin{lem} \cite[Theorem 5.4]{em53} \label{lem-shuffle-combinatorics}
Let $A,B$ and $C$ simplicial $k$-modules. Then

$\bullet$ $\sh_{A\otimes B,C}\circ (\sh_{A,B}\otimes\id)=\sh_{A,B\otimes C}\circ(\id\otimes\sh_{B,C})$, i.e. the shuffle maps are associative.

$\bullet$ For $a\in A_p$ and $b\in B_q$ and $\tau$ denoting the twist morphism, we have $\sh_{A,B}(a\otimes b) = (-1)^{pq} \tau_* \sh_{B,A}(b\otimes a)\tau^{-1}$, i.e. the shuffle maps are graded symmetric.
\end{lem}

\begin{lem} \cite[Theorem 2.1]{em54}
The shuffle and Alexander-Whitney maps are mutually inverse natural homotopy equivalences.
\end{lem}

\begin{lem} \cite[Corollary 2.2]{em54} \label{lem-AW-associative}
The Alexander-Whitney map is associative, i.e.~for $A,B$ and $C$ simplicial $k$-modules, the morphisms $(\id\otimes AW_{B,C})\circ AW_{A,B\hat{\otimes}C}$ and $(AW_{A,B}\otimes \id)\circ AW_{A\hat{\otimes} B,C}$ from $N(A\hat{\otimes} B\hat{\otimes} C)$ to $N(A)\otimes N(B)\otimes N(C)$ are equal.
\end{lem}
\begin{proof}
Let $a\otimes b \otimes c \in A_n\otimes B_n\otimes C_n$. For brevity, we write $\tilde{d}^n_i=d_{i+1}... d_{n-1}d_{n}$. Then the two compositions
$$N(A\hat{\otimes} B\hat{\otimes} C)\rightarrow N(A)\otimes N(B)\otimes N(C)$$
are
$$(a\otimes b \otimes c) \xmapsto{AW_{A,B\hat{\otimes}C}} \sum_{p=0}^n \tilde{d}^n_p a \otimes d_0^p b\otimes d_0^p c$$
$$\xmapsto{\id\otimes AW_{B,C}} \sum_{p=0}^n \sum_{s=0}^{n-p} \tilde{d}^n_p a \otimes \tilde{d}^{n-p}_s d_0^{p} b \otimes d_0^{p+s} c$$
and
$$(a\otimes b \otimes c) \xmapsto{AW_{A\hat{\otimes} B,C}} \sum_{q=0}^{n} \tilde{d}^n_q a \otimes \tilde{d}^n_q b \otimes d_0^{q} c$$
$$\xmapsto{AW_{B,C}\otimes \id} \sum_{q=0}^n \sum_{t=0}^q \tilde{d}^{q}_t \tilde{d}^n_q a \otimes d_0^{t} \tilde{d}^n_q b \otimes d_0^{q}c.$$
Note that $\tilde{d}^q_t\tilde{d}^n_q=\tilde{d}^n_t$. Using the simplicial identity $d_id_j=d_{j-1}d_i$ when $i<j$, observe that $d_0^{t} \tilde{d}^n_q = \tilde{d}^{n-t}_{q-t} d_0^{t}$. Writing $(q,t)=(p+s,p)$, we now see that the two expressions are equal.
\end{proof}

\bigskip To summarize the above theorems, the shuffle and Alexander-Whitney maps are mutually inverse quasi-isomorphisms. In particular, $AW\circ \sh=\id$ and $\sh\circ AW\simeq \id$. The shuffle map is a lax symmetric monoidal transformation witnessing that the normalized Moore complex functor $N\colon \sMod_k\rightarrow \Ch_k$, and hence also the Hochschild chains functor $C\colon \sAlg_k\rightarrow \Ch_k$ is lax symmetric monoidal. The Alexander-Whitney map is an oplax monoidal transformation witnessing that $N$, and hence $C$ is oplax monoidal. However, the Alexander-Whitney map is {\em not} symmetric. Still, it is $\mathbb{E}_\infty$ in the following sense (see Lemma \ref{lem-AW-E-infty}).

\begin{defn} (\cite[p.552]{richter00})
A functor $F\colon \C\rightarrow \D$ between symmetric monoidal categories is \emph{$\mathbb{E}_\infty$-monoidal} if there is an $\mathbb{E}_\infty$ operad $\mathcal{O}$ in $\D$ and maps
$$\mu_n\colon \mathcal{O}(n)\otimes \left(F(A_1)\otimes ... \otimes F(A_n)\right)\rightarrow F(A_1\otimes ... \otimes A_n)$$
such that
\begin{enumerate}
\item the action is unital, i.e.~if $I$ denotes the monoidal unit of $\D$ and $\eta\colon I\rightarrow \mathcal{O}(1)$ is the unit of the operad, then the following diagram commutes:
\tri{I\otimes F(A),\mathcal{O}(1)\otimes F(A), F(A),\eta\otimes \id,\mu_1,\simeq}
\item The action is equivariant: for each $\sigma\in\Sigma_n$, the action $\mu_n$ is compatible with the action of $\Sigma_n$ on $\mathcal{O}(n)$ and by permuting indices of the $A_i$. I.e.~the following diagram commutes:
\end{enumerate}
\begin{center}
\begin{tikzpicture}
  \matrix (m) [matrix of math nodes,row sep=3em,column sep=2em,minimum width=2em,
  text height=1.5ex, text depth=0.25ex]
  {
     \mathcal{O}(n)\otimes F(A_1)\otimes ... \otimes F(A_n) & F(A_1\otimes ... \otimes A_n) \\
      \mathcal{O}(n)\otimes F(A_{\sigma^{-1}(1)})\otimes ... \otimes F(A_{\sigma^{-1}(n)}) & F(A_{\sigma^{-1}(1)}\otimes ... \otimes A_{\sigma^{-1}(n)}) \\};
  \path[-stealth]
    (m-1-1) edge node [auto] {$\mu_n$} (m-1-2)
    		edge node [auto] {$\sigma\otimes \sigma$} (m-2-1)
    (m-1-2) edge node [auto] {$F(\sigma)$} (m-2-2)
    (m-2-1) edge node [auto] {$\mu_n$} (m-2-2);
\end{tikzpicture}
\end{center}
\begin{enumerate}
\setcounter{enumi}{2}
\item The action is associative, i.e.~is compatible with the operad multiplication.
\end{enumerate}
$\mathbb{E}_\infty$-comonoidal functors are similarly defined by using structure maps
$$\nu_n\colon \mathcal{O}(n)\otimes F(A_1\otimes ... \otimes A_n)\rightarrow F(A_1)\otimes ... \otimes F(A_n).$$
\end{defn}

We will now define chain complexes which assemble into a dg-operad (and later a symmetric monoidal dg-category) witnessing that $AW$ is an $\mathbb{E}_{\infty}$-comonoidal transformation. 

\begin{defn}
Define the functors
$$N^{\hat{\otimes}n},N^{\otimes n}\colon \sMod_k^{\times n}\rightarrow \Ch_k$$
given by
$$N^{\hat{\otimes}n}(A_1,...,A_n)=N(A_1\hat{\otimes}...\hat{\otimes} A_n)$$
$$N^{\otimes n}(A_1,...,A_n)=N(A_1)\otimes ... \otimes N(A_n)$$
and let
$$\mathcal{O}(n):=\Nat_{\sMod_k^{\times n}}(N^{\hat{\otimes}n},N^{\otimes n})$$
\end{defn}

\begin{notn} \label{rmk-N-2cat}
Since the elements of $\mathcal{O}(n)$, and of the complex $\tilde{N}^{\Sigma}((n),(n))$ which we define below, are natural transformations, we can in particular view them as 2-morphisms in the 2-category of categories, and so 2-categorical constructions, like horizontal composition, can be applied to them. For a 4-tuple of morphisms $f,f'\colon a\rightarrow b$ and $g,g'\colon b\rightarrow c$ and a pair of 2-morphisms $\alpha\colon  f\rightarrow f'$ and $\beta\colon g\rightarrow g'$, we write $\beta * \alpha$ for their horizontal composition $\beta*\alpha\colon  gf\rightarrow g'f'$.
\end{notn}

\begin{lem} \label{lem-AW-E-infty}
The complexes $\mathcal{O}(n)$ assemble into an $\mathbb{E}_{\infty}$ operad witnessing that $N^{\otimes n}$ and $N^{\hat{\otimes n}}$ are $\mathbb{E}_\infty$-comonoidal functors and that $AW\colon N^{\otimes 2}\rightarrow N^{\hat{\otimes}2}$ is an $\mathbb{E}_\infty$-comonoidal transformation.
\end{lem}
\begin{proof}
Let $n_1+...+n_i=n$ be natural numbers. The operad structure on $\mathcal{O}$ is given by the maps
$$\mathcal{O}(i)\otimes \left(\mathcal{O}(n_1) \otimes ... \otimes \mathcal{O}(n_i) \right) \rightarrow \mathcal{O}(n)$$
given by $(\phi, \gamma_1 , ..., \gamma_i)\mapsto \phi \circ (\gamma_1 * ... * \gamma_i)$. The $\Sigma_n$-action is given by conjugation, i.e. for $\chi\in \Sigma_n$ and $\psi\in \mathcal{O}(n)$ we have $\chi\cdot \psi = \chi\circ \psi \circ \chi^{-1}$. It is known (see \cite[Satz 1.6]{dold61}) that the complex of natural transformations 
$$\mathcal{O}(n)=\Nat_{\sMod_k^{\times n}}(N^{\hat{\otimes} n},N^{\otimes n})$$
is acyclic with zero'th homology $k$. It follows (see \cite[Section 7]{richter00} and \cite[Section 5]{richter03}) that the functors $N$ is an $\mathbb{E}_\infty$-comonoidal functor and that $AW\colon N^{\otimes 2}\rightarrow N^{\hat{\otimes}2}$ is an $\mathbb{E}_\infty$-comonoidal transformation.
\end{proof}

We will look at the complex
$$\widetilde{N}((n),(n)):=\Nat_{\sMod_k^{\times n}}(N^{\hat{\otimes} n},N^{\hat{\otimes} n})$$
which is homotopy equivalent to $\Nat_{\sMod_k^{\times n}}(N^{\hat{\otimes} n},N^{\otimes n})$, seen by post-composing with shuffle and Alexander-Whitney maps, but with the difference that maps in $\tilde{N}((n),(n))$ may be composed, giving rise to an algebra structure. In the rest of this section, we will construct a dg-category with morphism complexes built from $\Nat_{\sMod_k^{\times n}}(N^{\hat{\otimes} n},N^{\hat{\otimes} n})$, and the notation is chosen with this in mind.

\begin{defn}
The symmetric group $\Sigma_n$ acts on $\sMod_k^{\times n}$ by
$$\chi(A_1,...,A_n)=(A_{\chi^{-1}(1)},...,A_{\chi^{-1}(n)}).$$
Let $\tilde{N}^{\Sigma}((n),(n))$ be the complex
$$\tilde{N}^{\Sigma}((n),(n)) = \bigoplus_{\chi\in \Sigma_n} \Nat_{\sMod^{\times n}}(N^{\hat{\otimes} n},N^{\hat{\otimes} n}\circ \chi)=: \bigoplus_{\chi\in \Sigma_n} \widetilde{N}^{\Sigma}_\chi((n),(n)).$$
\end{defn}

\begin{lem} \label{lem-tildeN-contract}
The chain complex $\tilde{N}^\Sigma((n),(n))$ admits a $\Sigma_n$-graded algebra structure and contracts to $k\Sigma_n$ in degree $0$.
\end{lem}
\begin{proof}
Let $f\in \tilde{N}^\Sigma_{\chi}((n),(n))$ and $g\in \tilde{N}^\Sigma_{\chi'}((n),(n))$. We treat $f$ and $g$ as 2-morphisms in the 2-category of dg-categories as in Remark \ref{rmk-N-2cat}. The product of $g$ and $f$ is given by $(g*\id_{\chi})\circ f\colon N^{\hat{\otimes n}}\rightarrow N^{\hat{\otimes n}}\circ (\chi'\chi)$, which may also be visualized by the pasting diagram
\begin{center}
\begin{tikzpicture}
\node (A) at (-2,0) {$\sMod_k^{\times n}$};
\node (B) at (2,-2) {$\Ch_k$};
\node (C) at (-2,-2) {$\sMod_k^{\times n}$};
\node (D) at (-2,-4) {$\sMod_k^{\times n}$};
\node at (-1,-1.3) {\rotatebox{270}{$\Rightarrow$}};
\node at (-0.7,-1.3) {$f$};
\node at (-1,-2.7) {\rotatebox{270}{$\Rightarrow$}};
\node at (-0.7,-2.7) {$g$};
\path[->,font=\scriptsize,>=angle 90]
(A) edge node[above right] {$N^{\hat{\otimes} n}$} (B)
edge node[auto] {$\chi$} (C)
(C) edge node[auto] {$N^{\hat{\otimes} n}$} (B)
edge node[auto] {$\chi'$} (D)
(D) edge node[below right] {$N^{\hat{\otimes}n}$} (B);
\end{tikzpicture}
\end{center}
This gives the graded algebra structure. As for the contraction, the components $\tilde{N}^{\Sigma}_\chi((n),(n))$ are isomorphic to $\tilde{N}((n),(n))$ by pre-composition by $\chi$ and $\chi^{-1}$. As $\tilde{N}((n),(n))$ contracts onto $\id_{N^{\hat{\otimes} n}}$, $\tilde{N}^{\Sigma}_\chi((n),(n))$ contracts similarly to $\chi$.
\end{proof}

\begin{defn}
\textbullet \ For $A_1,...,A_k$ simplicial $k$-modules, we introduce the shorthand
$$N^{(k_1,...,k_n)}(A_1,...,A_k)=N(A_1\hat{\otimes}...\hat{\otimes}A_{k_1})\otimes ... \otimes N(A_{k_1+...+k_{n-1}+1}\hat{\otimes}...\hat{\otimes}A_k)$$
where $k=k_1+...+k_n$. Let $m_1+...+m_l=k$. Writing $\vec{k}=(k_1,...,k_n)$ and similarly for $\vec{m}$, define the complex
$$\tilde{N}(\vec{k},\vec{m}):=\Nat_{\sMod_k^{\times k}}(N^{\vec{k}},N^{\vec{m}})$$
Its symmetrized version $\tilde{N}^{\Sigma}(\vec{k},\vec{m})$ is defined as before by
$$\tilde{N}^{\Sigma}(\vec{k},\vec{m})=\bigoplus_{\chi\in \Sigma_k} \Nat_{\sMod_k^{\times k}}(N^{\vec{k}},N^{\vec{m}}\circ \chi) =: \bigoplus_{\chi\in \Sigma_k} \tilde{N}^{\Sigma}_{\chi}(\vec{k},\vec{m})$$

\noindent \textbullet \ We will refer to a finite sequence of integers $\vec{k}=(k_1,...,k_n)$ as a \emph{vector}. The sum of the entries of a vector is called its \emph{length} and denoted by $|\vec{k}|:=k_1+...+k_n$.

\noindent \textbullet \ We write $\tilde{N}^\Sigma$ for the dg-category whose objects are vectors $\vec{k}$, and whose morphism complexes are given by the $\tilde{N}^\Sigma(\vec{k},\vec{m})$ defined above.

\end{defn}

\begin{notn} \label{notn-sh-aw}
For any $\vec{k}=(k_1,...,k_n)$ with $|\vec{k}|=k$, by Lemma \ref{lem-shuffle-combinatorics} composing shuffle maps gives rise to a well-defined shuffle map which we denote by $\sh_{\vec{k}}\colon N^{\vec{k}}\rightarrow N^{(k)}=N^{\hat{\otimes}k}$. Similarly, the Alexander-Whitney map is associative by Lemma \ref{lem-AW-associative}, so composing AW-maps gives rise to a well-defined map $AW_{\vec{k}}\colon N^{(k)}\rightarrow N^{\vec{k}}$. Note that $AW_{\vec{k}}\circ \sh_{\vec{k}}\simeq\id_{N^{\vec{k}}}$ and $\sh_{\vec{k}}\circ AW_{\vec{k}}\simeq \id_{N^{(k)}}$.
\end{notn}

\begin{lem} \label{lem-doldkan-contractibility-1}
For every pair $\vec{k}=(k_1,...,k_n),\vec{m}=(m_1,...,m_l)$, the assignment
$$\phi\colon  f\mapsto AW_{\vec{m}}\circ f \circ \sh_{\vec{k}}$$
defines a homotopy equivalence $\tilde{N}^{\Sigma}((n),(n))\rightarrow \tilde{N}^{\Sigma}(\vec{k},\vec{m})$ with homotopy inverse
$$\psi\colon  g\mapsto \sh_{\vec{m}}\circ g \circ AW_{\vec{k}}$$
In particular, $\tilde{N}^{\Sigma}(\vec{k},\vec{m})$ contracts onto the degree zero subcomplex of elements of the form $AW_{\vec{m}}\circ \chi_*\circ \sh_{\vec{k}}$ for some $\chi\in \Sigma_k$.
\end{lem}
\begin{proof}
Fix homotopies $\alpha_{\vec{k}}\colon AW_{\vec{k}}\sh_{\vec{k}}\rightarrow \id$ and $\beta_{\vec{k}}\colon \sh_{\vec{k}}AW_{\vec{k}}\rightarrow \id$. Then we get homotopies
$$\beta_{\vec{m}}*\id*\beta_{\vec{k}}\colon  \psi\circ\phi \rightarrow \id$$
$$\alpha_{\vec{m}}*\id*\alpha_{\vec{k}}\colon  \phi\circ\psi \rightarrow \id$$
so that $\phi$ and $\psi$ are mutually inverse homotopy equivalences. Now the composition
$$k\Sigma_n \hookrightarrow \tilde{N}^{\Sigma}((n),(n))\rightarrow \tilde{N}^{\Sigma}(\vec{k},\vec{m})$$
takes $\chi$ to $AW_{\vec{m}}\circ \chi_*\circ \sh_{\vec{k}}$ and is a homotopy equivalence since $k\Sigma_n \hookrightarrow \tilde{N}^{\Sigma}((n),(n))$ is by Lemma \ref{lem-tildeN-contract}. The inverse
$$\tilde{N}^{\Sigma}(\vec{k},\vec{m})\rightarrow \tilde{N}^{\Sigma}((n),(n))\rightarrow k\Sigma_n$$
sends $AW_{\vec{m}}\circ \chi_*\circ \sh_{\vec{k}}$ to $\chi$, so $\tilde{N}^{\Sigma}(\vec{k},\vec{m})$ contracts as claimed.
\end{proof}

We now turn to establishing the action of $\tilde{N}^{\Sigma}$ on Hochschild complexes of dg-algebras. We begin by constructing a way of differentially extending functors between additive categories.

\begin{construction} \label{const-epsilon}
Let $\A$ and $\B$ be additive categories and let $\Fun^{pt}(\A,\B)$ be the category of pointed functors between them, that is, functors $F\colon \A\rightarrow \B$ such that there is an isomorphism $F(0)\simeq 0$. Note that $\Fun^{pt}(\A,\B)$ is itself an additive category. Denote by $m\text{-}\Ch(B)$ the additive category of $m$-fold chain complexes in $\B$. We will produce an additive functor
$$ (-)_\epsilon \colon  \Fun^{pt}(\A^{\times n},m\text{-}\Ch(\B))\rightarrow \Fun^{pt}(\Ch(\A)^{\times n},(n+m)\text{-}\Ch(\B)).$$
Let $F\colon \A^{\times n}\rightarrow m\text{-}\Ch(\B)$ be a ponited functor. Then $F_\epsilon$ sends an $n$-tuple of chain complexes $(A^1,...,A^n)$ in $\A$ to the $(n+m)$-fold chain complex in $\B$ given in multidegree $(p_1,...,p_n,q_1,...,q_m)$ by $F(A^1_{p_1},...,A^n_{p_n})_{q_1,...,q_m}$. In the same multidegree, the differentials are given by
$$d_i=\left\{ \begin{array}{rl} (-1)^{p_1+...+p_{i-1}} d^{A^i} & \text{, } 1\leq i\leq n \\ (-1)^{p_1+...+p_n + q_1+...+q_{i-n-1}} d^{F(A^1,...,A^n)}_{i-n} & \text{, } n+1\leq i \leq n+m \end{array} \right. $$
Since $F$ is a pointed functor, this does indeed define an $(n+m)$-fold chain complex in $\B$. Similarly, $F_\epsilon$ sends an $n$-tuple of morphisms $(f^i\colon A^i\rightarrow B^i)_{1\leq i \leq n}$ to the morphism given on the first $n$ multidegrees $(p_1,...,p_n)$ by the morphism
$$F(f^1_{p_1},...,f^n_{p_n})\colon  F(A^1_{p_1},...,A^n_{p_n})\rightarrow F(B^1_{p_1},...,B^n_{p_n})$$
Hence $F_\epsilon$ is indeed a functor.

Now let $F,G\colon \A^{\times n}\rightarrow m\text{-}\Ch(\B)$ be pointed functors and let $\alpha\colon F\rightarrow G$ be a natural transformation. Then $\alpha_\epsilon$ is the natural transformation given by applying $\alpha$ levelwise, i.e. for an $n$-tuple of chain complexes $(A^1,...,A^n)$ in $\A$, the morphism
$$(\alpha_\epsilon)_{(A^1,...,A^n)}\colon F_\epsilon(A^1,...,A^n)\rightarrow G_\epsilon(A^1,...,A^n)$$
is given in the first $n$ multidegrees $(p_1,...,p_n)$ by the morphism
$$\alpha_{(A^1_{p_1},...,A^n_{p_n})}\colon F(A^1_{p_1},...,A^n_{p_n})\rightarrow G(A^1_{p_1},...,A^n_{p_n}).$$
With this definition it is clear that $(-)_\epsilon$ is an additive functor.
\end{construction}

\begin{defn}
Building on Construction \ref{const-epsilon} we define the functor
$$ (-)_\delta \colon  \Fun^{pt}(\A^{\times n},m\text{-}\Ch(\B))\rightarrow \Fun^{pt}(\Ch(\A)^{\times n},\Ch(\B))$$
as the composition of $(-)_\epsilon$ with the totalization functor $(n+m)\text{-}\Ch(\B)\rightarrow \Ch(\B)$.
\end{defn}

Before considering the monoidality properties of the functor $(-)_\delta$, we need an observation about totalizations of $n$-fold chain complexes.

\begin{obs} \label{obs-reorder}
Let $\A$ be an additive category. The symmetric group on $n$ letters acts on the category of $n$-fold chain complexes in $\A$ by reordering the differentials. Specifically, if $A$ is an $n$-fold chain complex in $\A$ and $\chi\in\Sigma_n$, we have
$$(\chi\cdot A)_{p_1,...,p_n} = A_{p_{\chi^{-1}(1)},...,p_{\chi^{-1}(n)}}$$
and the differentials are similarly reordered. Let $\chi_{(p_1,...,p_n)}$ be the image of $\chi$ under the blow-up homomorphism $\Sigma_n\rightarrow \Sigma_{p_1+...+p_n}$. There is a natural transformation $g_\chi\colon \Tot\rightarrow \Tot\circ\chi$ given in multidegree $(p_1,...,p_n)$ by the sign of $\chi_{(p_1,...,p_n)}$. Composition of these transformations has the same effect as applying the sign associated to the composite permutation, such that $(g_{\chi'}*\id_{\chi})\circ g_{\chi} = g_{\chi'\chi}$ for any pair $\chi,\chi'\in \Sigma_n$.
\end{obs}

\begin{lem} \label{lem-delta-tensor}
Let the pointed functor $F\colon \A^{\times n}\rightarrow \Ch(\B)$ be given by $F(A^1,...,A^n)=F'(A^1,...,A^i)\otimes F''(A^{i+1},...,A^n)$ for a pair of pointed functors $F'\colon \A^{\times i}\rightarrow \Ch(\B)$ and $F''\colon \A^{\times n-i}\rightarrow \Ch(\B)$. Then there is a natural isomorphism
$$F_\delta (A^1,...,A^n)\simeq F'_\delta (A^1,...,A^i) \otimes F''_\delta(A^{i+1},...,A^n).$$
Furthermore, this isomorphism is associative.
\end{lem}
\begin{proof}
The tensor product $F'(A^1,...,A^i)\otimes F''(A^{i+1},...,A^n)$ is the totalization of a bicomplex, so we can lift $F_\delta$ to a functor $F_\epsilon\colon  \sCh_k^{\times n}\rightarrow (n+2)\text{-}\Ch_k$ such that $\Tot\circ F_\epsilon = F_\delta$. As before, the differentials in multidegree $(p_1,...,p_n,q_1,q_2)$ are given by
$$d_i=\left\{ \begin{array}{ll} (-1)^{p_1+...+p_{i-1}} d^{A^i} & \text{, } 1\leq i\leq n \\ (-1)^{p_1+...+p_n} d^{F'(A^1,...,A^i)} & \text{, } i=n+1 \\ (-1)^{p_1+...+p_n+q_1} d^{F''(A^{i+1},...,A^n)} & \text{, } i=n+2 \end{array} \right. $$
Now the tensor product $F'_\delta (A^1,...,A^i) \otimes F''_\delta(A^{i+1},...,A^n)$ is obtained by reordering the differentials and totalizing. Specifically, we must pass $d_{n+1}$ past $d_{i+1},...,d_{n}$, which in multidegree $(p_1,...,p_n,q_1,q_2)$ incurs a sign $(-1)^{q_1(p_{i+1}+...+p_n)}$. The natural isomorphism in the statement of the lemma is thus obtained as the transformation $g_{\chi_i}$ of Observation \ref{obs-reorder} where $\chi_i$ is the cycle $(i+1,...,n,n+1)\in \Sigma_{n+2}$.

The above recipe generalizes readily to a version with more than two tensor factors by replacing $\chi_i$ with the permutation $\chi_{i_1,...,i_m}\in \Sigma_{n+m}$ given by the composition of cycles
$$\chi_{i_1,...,i_m}=(i_1+...+i_m+1,...,n+m)\circ...\circ(i_1+1,...,n+1)$$
To see that this is associative, it is sufficient to look at the case of three factors:
$$F(A_1,...,A_{n})=F^1(A_1,...,A_i)\otimes F^{2}(A_{i+1},...,A_{i+j})\otimes F^3(A_{i+j+1},...,A_{n})$$
Associativity of the natural isomorphism above now follows from the identity
$$(i+j+2,...,n+1,n+2) \circ (i+1,...,n,n+1)$$
$$ = (i+1,...,n+1,n+2) \circ (i+j+1,...,n+1,n+2)$$
in $\Sigma_{n+2}$.
\end{proof}

\begin{defn}
Write $\zeta_n$ for the natural isomorphism
$$\zeta_n\colon (F^1\otimes ... \otimes F^n)_\delta \xrightarrow{\sim} F^1_\delta \otimes ...\otimes F^n_\delta$$
of functors $\A^{\times i}\rightarrow \Ch(\B)$ given by Lemma \ref{lem-delta-tensor}. 
\end{defn}

\begin{notn} \label{notn-C-vec}
Recall the cyclic bar construction from Definition \ref{defn-cyclic-bar-construction} (2). We extend this definiton to the functor $C^{\vec{k}}\colon \dgAlg_k\rightarrow \Ch_k$ by
$$A\mapsto C^{\vec{k}}(A)=C(A^{\otimes k_1})\otimes ... \otimes C(A^{\otimes k_n})$$
i.e. $C^{\vec{k}}(A)=N_\delta B^{cy}(A^{\otimes k_1})\otimes ... \otimes N_\delta B^{cy}(A^{\otimes k_n})$. 
\end{notn}

\begin{defn} \label{defn-theta-iso} Recall that the cyclic bar construction $B^{cy}\colon \dgAlg_k\rightarrow \sCh_k$ is a symmetric monoidal functor. We denote the natural structure isomorphism by
$$\theta_n\colon B^{cy}(A_1)\hat{\otimes} ... \hat{\otimes} B^{cy}(A_n)\rightarrow B^{cy}(A_1\otimes ... \otimes A_n)$$
The isomorphism is given in simplicial degree $k-1$ (in which we have $nk$ tensor factors) by the permutation $\chi_{n,k}\in \Sigma_{nk}$ sending $i+dk$ to $d+1+(i-1)n$ for $0< i \leq k$ and $0\leq d <n$, with a sign like that in the proof of Lemma \ref{lem-bcyc-sym-mon}.
\end{defn}

\begin{prop} \label{defn-action-hochschild}
The dg-category $\tilde{N}^{\Sigma}$ acts on Hochschild complexes of dg-algebras. That is, we have natural transformations
$$\tilde{N}^{\Sigma}(\vec{k},\vec{m})\otimes C^{\vec{k}}\rightarrow C^{\vec{m}}$$
of functors $\dgAlg_k\rightarrow \Ch_k$ compatible with composition.
This action exhibits $C\colon \dgAlg_k\rightarrow \Ch_k$ as a symmetric monoidal, $\mathbb{E}_{\infty}$-comonoidal functor.
\end{prop}
\begin{proof}
For $\vec{k}=(k_1,...,k_n)$ denoting
$$N_\delta\theta_{\vec{k}}=(N_\delta (\theta_{k_1})\otimes ... \otimes N_\delta (\theta_{k_n}))\circ \zeta_n\colon $$
$$N^{\vec{k}}_\delta(\Delta_k B^{cy}(A)) \rightarrow N_\delta B^{cy}(A^{\otimes k_1})\otimes ... \otimes N_\delta B^{cy}(A^{k_n})=C^{\vec{k}}(A)$$
where $\Delta_k\colon \sCh\rightarrow \sCh^{\times k}$ is the diagonal functor.
If $f\in \tilde{N}^{\Sigma}$, we write $(f_\delta)_{(A_1,...,A_n)}$ for the component of $f_\delta$ at the $n$-tuple $(A_1,...,A_n)$ of simplicial chain complexes over $k$. We now have a composite morphism
\begin{center}
\begin{tikzpicture}
\node (A) at (-6,0) {$C^{\vec{k}}(A)$};
\node (B) at (0,0) {$N^{\vec{k}}_\delta (B^{cy}(A),...,B^{cy}(A))$};
\node (C) at (0,-2) {$N_\delta^{\vec{m}}(B^{cy}(A),...,B^{cy}(A))$};
\node (D) at (-6,-2) {$C^{\vec{m}}(A)$};

\path[->,font=\scriptsize,>=angle 90]
(A) edge node[above] {$N_\delta (\theta_{\vec{k}})^{-1}$} (B)
(B) edge node[auto] {$(f_\delta)_{(B^{cy}(A),...,B^{cy}(A))}$} (C)
(C) edge node[below] {$N_\delta (\theta_{\vec{m}})$} (D);
\end{tikzpicture}
\end{center}
We therefore get a morphism
$$i_{\vec{k},\vec{m}}\colon  \tilde{N}^{\Sigma}(\vec{k},\vec{m}) \rightarrow \Nat(C^{\vec{k}}(A), C^{\vec{m}}(A))$$
$$f\mapsto N_\delta\theta_{\vec{m}}\circ (f_\delta*\id_{\Delta_k\circ B^{cy}})\circ (N_\delta\theta_{\vec{k}})^{-1}$$
whose adjoint is the morphism in the statement of the proposition. The contractibility of $\tilde{N}$ (see Lemma \ref{lem-doldkan-contractibility-1}) now implies that $C\colon \dgAlg_k\rightarrow \Ch_k$ is $\mathbb{E}_\infty$-monoidal and $\mathbb{E}_\infty$-comonoidal. However, since the shuffle maps are strictly symmetric, it is in fact symmetric monoidal as claimed.

\end{proof}

\begin{lem}
The images of $\tilde{N}^\Sigma_{\chi}((n),(n))$ and $\tilde{N}^\Sigma_{\chi'}((n),(n))$ in $\End_{\Ch_k}(C((-)^{\otimes n}))$ are disjoint for $\chi\neq\chi'$.
\end{lem}
\begin{proof}
Note that for $A$ a commutative algebra we have
$$H_0(C((-)^{\otimes n}))=HH_0((-)^{\otimes n})=(-)^{\otimes n}$$
so the induced action of $f\in \tilde{N}^\Sigma((n),(n))$ on $H_0(C((-)^{\otimes n}))$ is given by permuting tensor factors. Namely, $\tilde{N}^\Sigma_{e}((n),(n))=\tilde{N}((n),(n))$ acts as the identity since each $f\in \tilde{N}((n),(n))$ is homotopic to the identity map. Now each $\tilde{N}^\Sigma_{\chi}((n),(n))$ is isomorphic to $\tilde{N}((n),(n))$, the map given by postcomposition by $\chi_*$, and it follows that each $f\in \tilde{N}^\Sigma_{\chi}((n),(n))$ is homotopic to $\chi_*$. In particular, for $A=k[x]$ we see that $\chi$ and $\chi'$ act differently on $k[x]^{\otimes n}\simeq k[x_1,...,x_n]$. This proves the claim.
\end{proof}

\section{DG-fattening of props}

In this section we will build the fattening functor for dg-props and prove Theorem A. The fattening functor will associate to a dg-prop $P$ a certain full subcategory of the free symmetric monoidal dg-category on $P$ and $\tilde{N}^{\Sigma}$, modulo relations expressing that the Dold-Kan morphisms are natural with respect to the morphisms of $P$.

\begin{rmk}
To spell out what Theorem A means, to each dg-prop $P$, there is a natural homotopy-coherent $P$-action on the Hochschild complex of $\Ass\otimes P$-algebras. The homotopies that make up the coherencies are encoded in a replacement dg-prop $\tilde{P}$ which strictify the homotopy-coherent $P$-action. This strictification is moreover functorial in the prop.
\end{rmk}

In order to produce the functor $\tilde{(-)}$, we first construct an auxiliary functor $Q\colon \mathsf{dgprop}\rightarrow \dgCat^{\otimes}$ landing in symmetric monoidal dg-categories. $Q$ is constructed using a natural family of generators and relations and will contain $\tilde{(-)}$ as a full subfunctor, i.e. there will be a natural transformation $\tilde{(-)}\rightarrow Q$ whose components are inclusions of full subcategories.

\begin{construction} \label{const-Q}

$\bullet$ Let $C$ be a dg-category. We define a symmetric monoidal dg-category $C^\otimes$ given as follows. The objects of $C^\otimes$ is the free monoid on the objects of $C$. Given two words $a=a_1 \otimes ... \otimes a_n$ and $b=b_1 \otimes ... \otimes b_n$ in $\Ob C^\otimes$, the morphism complex $a\rightarrow b$ in $C^\otimes$ is given by
$$\Hom_{C^\otimes}(a,b) = \bigoplus_{\sigma\in \Sigma_n} \bigotimes_{i=1}^n \Hom_C (a_i,b_{\sigma(i)}).$$
We write $(f_1\otimes ... \otimes f_n)_\sigma$ for an elementary tensor in the summand of $\sigma\in\Sigma_n$. Note that for the empty product $\emptyset\in \Ob C^\otimes$, the tensor product is indexed over the empty set, which by convention means that $\Hom_{C^\otimes}(\emptyset,\emptyset)=k$ concentrated in degree 0.

If $a$ and $b$ are words of different lengths, i.e. $a=a_1\otimes ... \otimes a_n$ and $b=g_1\otimes ... \otimes b_m$ with $n\neq m$, then $\Hom_{C^\otimes}(a,b)=0$.

The composition
$$\Hom_{C^\otimes}(b,c)\otimes \Hom_{C^\otimes}(a,b) \rightarrow \Hom_{C^\otimes}(a,c)$$
is given on elementary tensors by sending $f\colon a\rightarrow b$ and $g\colon b\rightarrow c$, where $f=(f_1\otimes ... \otimes f_n)_\sigma$ and $g=(g_1\otimes ... \otimes g_n)_{\sigma'}$, to
$$g\circ f=(-1)^{\sgn(f,g,\chi)}( g_{\sigma(1)}\circ f_1 \otimes ... \otimes g_{\sigma(n)}\circ f_n)_{\sigma'\sigma}$$
where $\chi$ is the permutation of
$$(g_1\otimes ... \otimes g_n \otimes f_1 \otimes ... \otimes f_n)$$
into
$$(g_{\sigma(1)}\otimes f_1 \otimes g_{\sigma(2)}\otimes f_2 \otimes ... \otimes g_{\sigma(n)}\otimes f_n)$$
and $\sgn(f,g,\chi)$ is the weighted sign of $\chi$.

The symmetric monoidal structure on $C^\otimes$ is given on objects by multiplication in the free monoid on $\Ob C$ and on morphisms by the inclusion
$$\Hom_{C^\otimes}(a,b)\otimes \Hom_{C^\otimes}(a',b')\rightarrow \Hom_{C^\otimes}(a\otimes a',b\otimes b')$$
$$\left(\bigoplus_{\sigma\in \Sigma_n} \bigotimes_{i=1}^n \Hom_C (a_i,b_{\sigma(i)})\right)\otimes \left( \bigoplus_{\sigma'\in \Sigma_m} \bigotimes_{j=1}^m \Hom_C (a'_j,b'_{\sigma'(j)})  \right)\rightarrow \left( \bigoplus_{\sigma''\in \Sigma_{n+m}} \bigotimes_{l=1}^{n+m} \Hom_C (a_l'',b_{\sigma''(l)}'')\right)$$
where the summand of $\sigma\in \Sigma_n, \sigma'\in \Sigma_m$ lands in the summand of
$$\sigma''=\sigma\times \sigma'\in \Sigma_n\times \Sigma_m\hookrightarrow \Sigma_{n+m}$$
through the canonical inclusion, and an elementary tensor
$$f_1\otimes ... \otimes f_n \otimes f'_1\otimes ... \otimes f'_m$$
is sent to
$$f''_1\otimes ... \otimes f''_{n+m},$$
where $f_l'' \colon  a_l'' \rightarrow b''_{\sigma''(l)}$ is given by
$$\begin{array}{ll} f_l \colon  a_l\rightarrow b_{\sigma(l)} & \text{, if } 1\leq l \leq n \\ f'_{l-n}\colon a'_{l-n}\rightarrow b'_{\sigma'(l-n)} & \text{, if } n+1\leq l \leq n+m \end{array} $$

We check that the monoidal product is functorial, i.e. we wish to verify that the following diagram commutes
. Consider morphisms
$$a^i\xrightarrow{f^i} b^i \xrightarrow{g^i} c^i \quad , \quad i=0,1$$
such that each $f^i,g^i$ is an elementary tensor
$$f^i=(f^i_1 \otimes ... \otimes f^i_n)_{\sigma_i}$$
$$g^i=(g^i_1\otimes ... \otimes g^i_{n})_{\sigma'_i}$$
and compare the operations
$$H_0(g^0,g^1,f^0,f^1):=(g^0\otimes g^1)\circ(f^0\otimes f^1)$$
and
$$H_1(g^0,g^1,f^0,f^1):=(g^0\circ f^0)\otimes (g^1\circ f^1).$$
From the recipes given for composition and monoidal product, we see that $H_0$ and $H_1$ are at least equal up to sign on elementary tensors. The signs are in both cases the weighted sign of the same permutation, so they are equal. It follows that the monoidal product is functorial. Note that the monoidal product is strictly associative and unital, where the unit for the monoidal structure is given by the empty product $\emptyset$.



For $a=a_1\otimes ... \otimes a_n$ and $b=b_1\otimes ... \otimes b_m$, the twist map $\tau\colon a\otimes b\rightarrow b\otimes a$ is given by the elementary tensor
$$(\id_{a_1} \otimes ... \otimes \id_{a_n} \otimes \id_{b_1} \otimes ... \otimes \id_{b_m})_{\tau_{n,m}}$$
where $\tau_{n,m}\in \Sigma_{n+m}$ is the block permutation permuting the first $n$ letters past the latter $m$ letters.

Observe that for any elementary tensor $f=f_1\otimes ... \otimes f_n\in \Hom_{C^\otimes}(a,b)$ and any $\sigma\in\Sigma_n$, conjugation of $f$ by $\sigma$ produces the elementary tensor
$$f_{\sigma(1)}\otimes ... \otimes f_{\sigma(n)}\in \Hom_{C^\otimes}(\sigma\cdot a ,\sigma\cdot b).$$
In particular, for morphisms $f\colon a\rightarrow b$ and $f'\colon a'\rightarrow b'$ in $C^\otimes$, the conjugation of $f\otimes f'$ by $\tau_{n,m}$ is precisely $f'\otimes f$, such that the twist morphism $\tau$ is a symmetry for the monoidal structure.

\bigskip $\bullet$ Let $\vec{\Sigma}$ be the category whose objects are vectors, and whose morphisms are permutations of the entries of vectors. Note that $\vec{\Sigma}$ is concentrated in degree 0, and can also be defined as $\N^\otimes$, the free symmetric monoidal dg-category on the discrete dg-category $\N$. Here by discrete we mean that $\Hom_\N(n,m)=k$ in degree 0 if $n=m$, and zero otherwise. Then there are symmetric monoidal functors $\vec{\Sigma}\rightarrow \tilde{N}^\Sigma$ and $\vec{\Sigma}\rightarrow P^\otimes$ which are the identity on objects.

\bigskip $\bullet$ By viewing the morphism complexes of $P^\otimes$ as being concentrated in horizontal degrees, we give $P^\otimes$ an enrichment in bicomplexes. Similarly, we view $\tilde{N}^\Sigma$ as being enriched in bicomplexes, concentrated in vertical degrees. We define a category enriched in bicomplexes $Q_0(P)$ as follows.

The objects of $Q_0(P)$ are vectors $\vec{k}=(k_1,...,k_n)$. To give a description of the morphism complexes $Q_0(P)(\vec{k},\vec{m})$, we define the following auxiliary notation. For $\vec{k},\vec{m}\in\Ob \, Q_0(P)$ and $j\in\N $, let $\C_{j,l}(\vec{k},\vec{m})$ be given by
$$\C_{j}(\vec{k},\vec{m})=\left\{ \begin{array}{ll} 
P^\otimes(\vec{k},\vec{m}) & \text{, if } j \text{ even}, \\
\tilde{N}^\Sigma (\vec{k},\vec{m}) & \text{, if } j \text{ odd}.
\end{array}\right. $$
Then $Q_0(P)(\vec{k},\vec{m})$ is given by the direct sum
$$Q_0(P)(\vec{k},\vec{m})=\left. \left( \bigoplus_{\substack{i\in\N \\ \vec{x}_1,...,\vec{x}_i\in \Ob Q_0(P) \\ l=0,1 }} \C_{i+l}(\vec{x}_i,\vec{m})\otimes_{\Sigma} \C_{i-1+l}(\vec{x}_{i-1},\vec{x}_i)\otimes_{\Sigma} ... \otimes_\Sigma \C_{l}(\vec{k},\vec{x}_1) \right) \right/ \sim_{red} ,$$
where the tensor product $\otimes_\Sigma$ means taking the colimit of the diagrams
$$\C_{j+l}(\vec{x}_j,\vec{x}_{j+1})\otimes \vec{\Sigma}(\vec{x}_j,\vec{x}_j)\otimes \C_{j-1+l}(\vec{x}_{j-1},\vec{x}_j) \ \substack{\longrightarrow \\ \longrightarrow} \ \C_{j+l}(\vec{x}_j,\vec{x}_{j+1})\otimes \C_{j-1+l}(\vec{x}_{j-1},\vec{x}_j),$$
where $\vec{\Sigma}$ acts via its inclusions into $P^\otimes$ and $\tilde{N}^\Sigma$ described above. The equivalence relation $\sim_{red}$ is generated as an equivalence relation by the following relations:
\begin{enumerate}
\item Permutation morphisms on the ends are absorbed into their neighbour. For example, the word
$$\vec{k}_0\xrightarrow{\chi} \vec{k}_1 \xrightarrow{\phi} \vec{k}_2$$
where $\chi\in\vec{\Sigma}$ and $\phi\in \tilde{N}^\Sigma$ is equivalent to the word
$$k_0 \xrightarrow{\chi\circ \phi} \vec{k}_2$$
using the inclusion $\vec{\Sigma}\rightarrow \tilde{N}^\Sigma$.
\item If a morphism in the middle of a word is a permutation morphism, then it is absorbed into either its left or right hand neighbour.
\item If two adjacent morphisms are either both from $P^\otimes$ of both from $\tilde{N}^\Sigma$, then they are composed.
\end{enumerate}
In other words a morphism in $P$ is a reduced word composed of morphisms in $P^\otimes$ and $\tilde{N}^\Sigma$, where the subcategories $\vec{\Sigma}\subset P^\otimes$ and $\vec{\Sigma}\subset \tilde{N}^\Sigma$ are identified. Composition is given by concatenating words and reducing, similarly to the case of free products of algebras. Note that when we consider a general morphism in $Q_0(P)$ which can be represented by an elementary tensor, we can add identity arrows at the ends as needed to ensure that the first arrow is from $P^\otimes$ and the last is from $\tilde{N}^\Sigma$. When we work with such morphisms later, we will implicitly choose a representative of this form.

As for the enrichment in bicomplexes, $(f)\in P(n,m)$ has bidegree $(|f|,0)$ and $\psi\in \tilde{N}^{\Sigma}(\vec{n},\vec{m})$ has bidegree $(0,|\psi|)$. The horizontal and vertical differentials act on $P$ and $\tilde{N}^{\Sigma}$ respectively. To be precise, the horizontal differential of a word
$$\vec{k}_{0}\xrightarrow{\phi^0}\vec{k}_{1} \xrightarrow{\gamma^0} \vec{k}_{2} \xrightarrow{\phi^1} ... \xrightarrow{\gamma^{m-1}} \vec{k}_{2m}$$
where $\phi^i\in P^{\otimes n_i}(\vec{k}_{2i},\vec{k}_{2i+1})$ and $\gamma^i \in \tilde{N}^\Sigma(\vec{k}_{2i+1},\vec{k}_{2i+2})$, is the sum of applications of the differential to each $\phi^i$, with sign $(-1)^{|\phi^0|+|\gamma^0|+...+|\gamma^{i-i}|}$, and the vertical differential is defined similarly.

There is one exceptional object in $Q_0(P)$, the empty vector $\vec{0}=()$. The morphism space $\Hom_{Q_0}(\vec{0},\vec{0})$ is declared to be the ground ring $k$ in bidegree $(0,0)$. Note that $\vec{0}$ does not admit any morphism to or from a non-empty vector, as generators in $\tilde{N}^\Sigma$ cannot change the length of a vector and generators in $P^\otimes$ cannot change the number of entries.

\bigskip $\bullet$ We write $*$ for the concatenation of vectors. For a morphism $f\colon \vec{k}_0\rightarrow\vec{k}_1$ in $P^\otimes$, and a vector $\vec{m}$, let
$$f\otimes^P\id\colon \vec{k}_0 * \vec{m}\rightarrow \vec{k}_1 * \vec{m}$$
be the tensor product of $f$ and $\id_{\vec{m}}$ using the symmetric monoidal structure of $P^\otimes$. The morphisms $\id \otimes^P f$ is defined similarly. For a morphism $\psi\colon \vec{k}_0\rightarrow \vec{k}_1$ in $\tilde{N}^\Sigma$, we also define
$$\psi\otimes^N \id\colon  \vec{k}_0 * \vec{m}\rightarrow \vec{k}_1 * \vec{m}$$
and similarly $\id\otimes^N \psi$ using the symmetric monoidal structure of $\tilde{N}^\Sigma$.

Let $Q_1(P)$ be the quotient of $Q_0(P)$ with respect to the following relations. For each pair of morphisms $f\colon \vec{k}_0\rightarrow \vec{k}_1$ in $P^\otimes$ and $\psi\colon \vec{m}_0\rightarrow \vec{m}_1$ in $\tilde{N}^\Sigma$, we impose an \emph{interchange relation}, i.e. that the following diagrams commute.
\twosq{\vec{k}_0*\vec{m}_0,\vec{k}_0*\vec{m}_1,\vec{k}_1*\vec{m}_0,\vec{k}_1*\vec{m}_1, \id\otimes^N \psi, f\otimes^P \id, f\otimes^P \id,\id\otimes^N \psi,\vec{m}_0*\vec{k}_0,\vec{m}_1*\vec{k}_0,\vec{m}_0*\vec{k}_1,\vec{m}_1*\vec{k}_1,\psi\otimes^N\id,\id\otimes^P f, \id\otimes^P f, \psi\otimes^N \id}
To be precise, $Q_1(P)$ is the quotient of $Q_0(P)$ with respect to the ideal
$$I=\left( (f\otimes^P \id) \circ (\id\otimes^N \psi) - (\id\otimes^N \psi)\circ(f\otimes^P \id), \right.$$
$$ (f\otimes^P \id)\circ (\id\otimes^N \psi) - (\id\otimes^N \psi)\circ(f\otimes^P \id) \ \left| \ f\in P^\otimes, \psi\in \tilde{N}^\Sigma \right)$$

Note that this relation is an identification of generators which respects bidegrees of morphisms and is compatible with concatenation of words. It also respects the differentials of morphisms, hence the result is a well-defined category enriched in bicomplexes. We expand this notation to general morphisms. Given a morphism $f\colon \vec{k}\rightarrow\vec{m}$ in $Q_1(P)$ represented by a word
$$\vec{k}=\vec{k}_{0}\xrightarrow{\phi^0}\vec{k}_{1} \xrightarrow{\gamma^0} \vec{k}_{2} \xrightarrow{\phi^1} ... \xrightarrow{\gamma^{m-1}} \vec{k}_{2m}=\vec{m}$$
where $\phi^i\in P^{\otimes n_i}(\vec{k}_{2i},\vec{k}_{2i+1})$ and $\gamma^i \in \tilde{N}^\Sigma(\vec{k}_{2i+1},\vec{k}_{2i+2})$, and a vector $\vec{l}$, we write $f\otimes_1 \id$ for the morphism represented by the word
$$(\vec{k}_{0}*\vec{l})\xrightarrow{\phi^0\otimes^P\id} (\vec{k}_1*\vec{l})\xrightarrow{\gamma^0\otimes^N \id} ... \xrightarrow{\gamma^{m-1}\otimes^N\id} (\vec{k}_{2m}*\vec{l})$$
We also define $\id\otimes_1 f$ similarly. Note that if $\vec{l}=\vec{0}$ then
$$f\otimes_1 \id = f = \id\otimes f.$$
To see that this construction is well-defined, take morphisms $\phi\colon \vec{k}_0\rightarrow \vec{k}_1$ and $\phi'\colon \vec{k}_1\rightarrow \vec{k}_2$ in $P^\otimes$ and vectors $\vec{m}_0,\vec{m}_1$. Then the following properties of the operation $-\otimes^P \id$ are immediate.
\begin{enumerate}
\item $\id\otimes^P (\phi'\circ \phi)= (\id\otimes^P \phi')\circ (\id\otimes^P \phi)$ and similarly for $(\phi'\circ\phi)\otimes^P \id)$,
\item for the identity morphism $\id_{\vec{k}}\colon \vec{k}\rightarrow \vec{k}$ in $P^\otimes$, we have $\id\otimes \id_{\vec{k}}=\id_{\vec{m}*\vec{k}}$ and $\id_{\vec{k}}\otimes \id=\id_{\vec{k}*\vec{m}}$ in $P^\otimes$,
\item $((\phi\otimes^P \id)\otimes^P \id)=\phi\otimes^P\id\colon  \vec{k}_0*\vec{m}_0*\vec{m_1}\rightarrow \vec{k}_1*\vec{m}_0*\vec{m_1}$,
\item $(\id\otimes^P (\phi \otimes^P \id))=((\id\otimes^P \phi)\otimes^P\id)\colon \vec{m_0}*\vec{k}_0*\vec{m}_1 \rightarrow \vec{m_0}*\vec{k}_1*\vec{m}_1$,
\item $d(\id \otimes^P f)=\id\otimes^P df$ and $d(f \otimes^P \id)=(df)\otimes^P \id$.
\end{enumerate}
Similar identities hold for morphisms in $\tilde{N}^\Sigma$. We regard property (2) as a definition if $\vec{k}=\vec{0}$. Together these properties imply that if $f,g$ are morphisms $\vec{k}_0\rightarrow \vec{k}_1$ in $Q_0(P)$ such that $f\sim g$ according to the interchange relation, then also $(f\otimes_1 \id)\sim (g\otimes_1 \id)$ and $(\id\otimes_1 f)\sim(\id\otimes_1 g)$, and that the operations $f\mapsto f\otimes^P\id$ and $f\mapsto \id\otimes^P f$, and the same operations for $\tilde{N}^\Sigma$ are well-defined morphisms of bicomplexes.

The interchange relation now implies that
$$(\id\otimes_1 f)\circ (g\otimes_1 \id)=(g\otimes_1 \id)\circ (\id\otimes_1 f)$$
for any pair of morphisms $f\colon \vec{k}_0\rightarrow \vec{k}_1$ and $g\colon \vec{m}_0\rightarrow \vec{m}_1$ in $Q_1(P)$.

Finally, observe that for a pair of morphisms
$$\vec{k}_0\xrightarrow{f_0}\vec{k}_1\xrightarrow{f_1}\vec{k}_2$$
in $Q_1(P)$ and a vector $\vec{m}$, we have
$$(f_1\circ f_0)\otimes_1 \id = (f_1\otimes_1 \id)\circ (f_0\otimes_1 \id)$$
and
$$\id\otimes_1 (f_1\circ f_0) = (\id\otimes_1 f_1)\circ (\id\otimes_1 f_0),$$
such that $(-\otimes_1 \id)$ and $(\id\otimes_1 -)$ are endofunctors on $Q_1(P)$.

\bigskip $\bullet$ Recall that for a pair of categories $\C$ and $\D$ enriched in bicomplexes, their tensor product $\C\otimes \D$ has objects $\Ob (\C\otimes \D)=\Ob\C\times \Ob \D$ and morphisms
$$\Hom_{\C\otimes \D}((a,b),(a',b'))=\Hom_{\C}(a,a')\otimes \Hom_{\D}(b,b').$$

We define a symmetric monoidal structure on $Q_1(P)$ as follows. The functor
$$\otimes_1\colon Q_1(P)\otimes Q_1(P)\rightarrow Q_1(P)$$
is given on objects by concatenating vectors. Given two morphisms $f\colon \vec{k}_0\rightarrow \vec{k}_1$ and $g\colon \vec{m}_0\rightarrow \vec{m}_1$ in $Q_1(P)$, we define the monoidal product
$$f\otimes_1 g\colon \vec{k}_0 * \vec{m}_0 \rightarrow \vec{k_1}*\vec{m_1}$$
by the formula
$$f\otimes_1 g := (\id\otimes_1 g)\circ (f\otimes_1 \id).$$
Since $(-\otimes_1 \id)$ is a well defined operation on morphisms in $Q_1(P)$, this gives rise to a well defined morphism of bicomplexes
$$\Hom_{Q_1(P)}(\vec{k}_0,\vec{k}_1)\otimes \Hom_{Q_1(P)}(\vec{m}_0,\vec{m}_1)\xrightarrow{-\otimes_1 -} \Hom_{Q_1(P)}(\vec{k}_0*\vec{m}_0,\vec{k}_1*\vec{m}_1).$$ 

We check that $\otimes_1$ is a functor. Given morphisms
$$\vec{k}_0 \xrightarrow{f_0}\vec{k}_1\xrightarrow{f_1} \vec{k}_2$$
and
$$\vec{m}_0\xrightarrow{g_0}\vec{m}_1\xrightarrow{g_1}\vec{m}_2,$$
we have
$$(f_1\circ f_0)\otimes_1 (g_1\circ g_0) = (\id\otimes_1 (g_1\circ g_0)\circ ((f_1\circ f_0)\otimes_1 \id)$$
$$ = (\id\otimes_1 g_1)\circ (\id\otimes_1 g_1) \circ (f_1\otimes_1 \id)\circ (f_0\otimes \id).$$
Now the interchange relation implies that
$$(\id\otimes_1 g_1) \circ (f_1\otimes_1 \id)= (f_1\otimes_1 \id)\circ(\id\otimes_1 g_1)$$
as morphisms in $Q_1(P)$, so that
$$(f_1\circ f_0)\otimes_1 (g_1\circ g_0) = (\id\otimes_1 g_1)\circ (f_1\otimes_1 \id)\circ (\id\otimes_1 g_1) \circ (f_0\otimes \id)$$
$$= (f_1\otimes_1 g_1)\circ (f_0\otimes_1 g_0)$$
as claimed.

The unit for the monoidal structure is given by the empty vector $\vec{0}=()$. 

Note also by properties (3) and (4) above that $\otimes_1$ is strictly associative, so $\otimes_1$ defines a strict monoidal structure on $Q_1(P)$.

To give a symmetry for $\otimes_1$, we first describe how the symmetric group $\Sigma_n$ acts on $Q_1(P)^{\otimes n}$. Given an $n$-tuple of vectors $(\vec{k}_1,...,\vec{k}_n)$, the permutation
$$(\sigma,\vec{k}_1,...,\vec{k}_n)\colon \vec{k}_1* ... * \vec{k}_n \rightarrow \vec{k}_{\sigma(1)}* ... * \vec{k}_{\sigma(n)}$$
is given by acting on the source by the block permutation $\sigma_{|\vec{k}_1|,...,|\vec{k}_n|}\in \Sigma_{|\vec{k}_1|+...+|\vec{k}_{n}|}$, i.e. if $\sigma(i)=j$, then for all $1\leq l\leq |\vec{k}_i|$, we have
$$\sigma_{|\vec{k}_1|,...,|\vec{k}_n|}\left( |\vec{k}_1|+...+|\vec{k}_{i-1}|+l \right) = |\vec{k}_1|+...+|\vec{k}_{j-1}|+l.$$
If $\phi\colon \vec{k}_0\rightarrow \vec{k}_1$ is a morphism in $P^\otimes$, $\vec{m}$ is a vector, and $\tau\in\Sigma_2$ is the twist permutation, then in $P^\otimes$ we have
$$(\id\otimes^P \phi)\circ \tau_{|\vec{k}_0|,|\vec{m}|} = \tau_{|\vec{k}_1|,|\vec{m}|}\circ (\phi\otimes^P\id)$$
and
$$(\phi\otimes^P \id)\circ \tau_{|\vec{m}|,|\vec{k}_0|} = \tau_{|\vec{m}|,|\vec{k}_1|}\circ (\id\otimes^P\phi)$$
and similarly for $\gamma\colon \vec{k}_0\rightarrow\vec{k}_1$ in $\tilde{N}^\Sigma$. This implies that for any morphism $f\colon \vec{k}_0\rightarrow \vec{k}_1$ in $Q_1(P)$, we have
$$(\id\otimes_1 f)\circ \tau_{|\vec{k}_0|,|\vec{m}|} = \tau_{|\vec{k}_1|,|\vec{m}|}\circ (f\otimes_1\id)$$
and
$$(f\otimes_1 \id)\circ \tau_{|\vec{m}|,|\vec{k}_0|} = \tau_{|\vec{m}|,|\vec{k}_1|}\circ (\id\otimes_1 f).$$
In particular, for $g\colon \vec{m}_0\rightarrow \vec{m}_1$ another morphism in $Q_1(P)$ we have
$$(g\otimes_1 f)\circ \tau_{|\vec{k}_0|,|\vec{m}_0|} = \tau_{|\vec{k}_1|,|\vec{m}_1|}\circ (f\otimes_1 g)$$
such that $\tau$ is a symmetry for $\otimes_1$.

\end{construction}

\begin{notn}
$\bullet$ We write $(1)^n$ for the vector $(1,...,1)$ of length $n$.

\bigskip $\bullet$ For a vector $\vec{a}=(a_1,...,a_n)$, we write $l(\vec{a})=n$ for its number of elements.
\end{notn}

\begin{rmk}
The following definition describes a way of functorially arranging the entries of a vector $\vec{k}=(k_1,...,k_n)$ (in particular, $l(\vec{k})=n$) according to the entries of a vector $\vec{a}$ of length $n$, which we use to define the functor $Q$. Informally one should think of $\Par_{\vec{k}}(\vec{a})$ as given by arranging the entries of $\vec{k}$ according to the entries of $\vec{a}$. Similarly, for a morphism $\gamma\colon \vec{a}\rightarrow \vec{b}$, one may think of $\Par_{\vec{k}}(\gamma)$ as the natural transformation $N^{\Par_{\vec{k}}(\vec{a})}\rightarrow N^{\Par_{\vec{k}}(\vec{b})}$ of functors $\sMod^{\times k}\rightarrow \Ch_k$ whose $(A_1,...,A_k)$-component equals the $(A_1\hat{\otimes} ... \hat{\otimes} A_{k_1} \, , \, ... \, , \, A_{k_1+...+k_{n-1}+1}\hat{\otimes} ... \hat{\otimes} A_k)$-component of $\gamma$.
\end{rmk}

\begin{defn} \label{defn-vector-partition}
For each vector $\vec{k}=(k_1,...,k_n)$, $k=|\vec{k}|$, let $\iota_{\vec{k}}\colon \sMod^{\times k}\rightarrow \sMod^{\times n}$ be the functor taking a $k$-tuple $(A_1,...,A_k)$ to the $n$-tuple $(B_1,...,B_n)$ where
$$B_i=A_{k_1+...+k_{i-1}+1}\hat{\otimes} ... \hat{\otimes} A_{k_1+...+k_i}.$$
Writing $N_n$ for the full subcategory of $\tilde{N}$ on the objects $\vec{a}$ with $|\vec{a}|=n$, let
$$\Par_{\vec{k}}\colon N_n\rightarrow N_k$$
be the functor taking $\vec{a}=(a_1,...,a_l)$ to
$$\Par_{\vec{k}}(\vec{a}):=(k_1+...+k_{a_1}\, , \, k_{a_1+1}+...+k_{a_1+a_2} \, , \, ... \, , \, k_{a_1+...+a_{l-1}+1} + ... + k_{n})$$
i.e.~the unique vector such that $N^{\vec{a}}\circ \iota_{\vec{k}} = N^{\Par_{\vec{k}}(\vec{a})}$. In particular we have $|\Par_{\vec{k}}(\vec{a})|=|\vec{k}|$ and $l(\Par_{\vec{k}}(\vec{a}))=l(\vec{a})$.
For a morphism $\gamma\colon \vec{a}\rightarrow \vec{b}$, $\Par_{\vec{k}}(\gamma)$ is given by $\gamma * \id_{\iota_{\vec{k}}}$. In particular, the following diagram commutes.
\begin{center}
\begin{tikzpicture}
\node (A) at (0,0) {$\sMod^{\times k}$};
\node (B) at (8,0) {$\sMod^{\times n}$};
\node (C) at (4,-3) {$\Ch_k$};

\node at (1.5,-1) {${\scriptsize \Par_{\vec{k}}(\gamma)}$};
\node at (2,-1.5) {\rotatebox{230}{$\Longrightarrow$}};
\node at (6.3,-1.2) {${\scriptsize \gamma}$};
\node at (6,-1.5) {\rotatebox{310}{$\Longrightarrow$}};

\path[->,font=\scriptsize,>=angle 90]
(A) edge [bend left=10] node[auto] {$\iota_{\vec{k}}$} (B)
(B) edge [bend left=30] node[auto]  {$N^{\vec{b}}$} (C)
(B) edge [bend right=30] node[above left]  {$N^{\vec{a}}$} (C)
(A) edge [bend left=30] node[auto] {$N^{\Par_{\vec{k}}(\vec{a})}$} (C)
(A) edge [bend right=30] node[below left] {$N^{\Par_{\vec{k}}(\vec{b})}$} (C);
\end{tikzpicture}
\end{center}

In particular, for any vector $\vec{a}$, we have $\Par_{\vec{a}}((1)^{l(\vec{a})})=\vec{a}$.
\end{defn}

\begin{construction} \label{const-Q-2}
$\bullet$ Let $\vec{a}=(a_1,...,a_l)$ be an object of $N_n$ (i.e. $|\vec{a}|=n$) and $f=f_1\otimes ... \otimes f_n\colon \vec{k}\rightarrow \vec{m}$ a morphism in $P^\otimes$, where $l(\vec{k})=l(\vec{m})=n$. For each $|\leq i \leq l$ we have a morphism
$$(k_{a_1+ ... a_i+1} + ... + k_{a_1+...a_{i+1}})\xrightarrow{(f_{a_1+ ... a_i+1} \otimes ... \otimes f_{a_1+...a_{i+1}})} (m_{a_1+ ... a_i+1} + ... + m_{a_1+...a_{i+1}})$$ 
in $P$, and we write $\Par_f(\vec{a})\colon \Par_{\vec{k}}(\vec{a})\rightarrow \Par_{\vec{m}}(\vec{a})$ for the tensor product in $P^\otimes$ of these morphisms for $1\leq i \leq l$. Similarly to Definition \ref{defn-vector-partition}, one can think of this as grouping the factors of $f$ according to the entries of $\vec{a}$. Below is an example where $\vec{a}=(1,3)$. In this example,
$$\Par_f(\vec{a})=\Par_f(\vec{a})_1\otimes^P \Par_f(\vec{a})_2$$
with
\[
  \begin{array}{@{} c @{}}
    \begin{array}{@{} r @{}}
      a_1=1 ~\{\hspace{\nulldelimiterspace} \\
      a_2=3 ~\left\{\begin{array}{@{}c@{}}\null\\\null\\\null\end{array}\right.
    \end{array}

      \begin{array}{ *{1}{c} }
        f_1\colon k_1\rightarrow m_1 \\
        f_2\colon k_2\rightarrow m_2 \\
        f_3\colon k_3\rightarrow m_3 \\
        f_4\colon k_4\rightarrow m_4 \\
      \end{array}

    \begin{array}{@{} l @{} l}
       \}\hspace{\nulldelimiterspace}~ & \Par_f(\vec{a})_1=f_1\colon k_1\rightarrow m_1 \\
       \left\}\begin{array}{@{}c@{}}\null\\\null\\\null\end{array}\right. ~ & \Par_f(\vec{a})_2=\bigotimes_{i=2}^4 f_i \colon \sum_{i=2}^4 k_i\rightarrow \sum_{i=2}^4 m_i.
    \end{array}
 \\
    \mathstrut
  \end{array}
\]

Let $Q(P)$ be the quotient of $Q_1(P)$ with respect to the ideal generated by the following relations. For each morphism $\gamma\colon \vec{a}\rightarrow \vec{b}$ in $N_n$ and each morphism $f\colon \vec{k}\rightarrow \vec{m}$ in $P^\otimes$ with $l(\vec{k})=l(\vec{m})=n$, the following diagram commutes:

\sq{\Par_{\vec{k}}(\vec{a}),\Par_{\vec{m}}(\vec{a}),\Par_{\vec{k}}(\vec{b}),\Par_{\vec{m}}(\vec{b}),\Par_{f}(\vec{a}),\Par_{\vec{m}}(\gamma),\Par_{\vec{k}}(\gamma),\Par_{f}(\vec{b})}

We call this the \emph{partition relation}. If two morphisms $f,g$ are equivalent under the partition relation, we write $f\sim_{\Par} g$. The partition relation captures in a more general manner the naturality of shuffle and Alexander-Whitney maps with respect to algebra homomorphisms, see Observation \ref{rmk-par-example}. 

These relations are compatible with concatenations of words and respect bidegrees of morphisms. They are also compatible with horizontal and vertical differentials as the relation only depends on the source and target of morphisms, hence the result is a well-defined category enriched in bicomplexes. 

To see that the symmetric monoidal structure is compatible with these relations, observe that
$$\Par_{\vec{k}}(\vec{a})*\vec{r}=\Par_{\vec{k}*\vec{r}}(\vec{a}*(1)^{l(\vec{r})}),$$

which implies that the following diagrams are equal and commute.

\begin{center}
\begin{tikzpicture}
  \matrix (m) [matrix of math nodes,row sep=3em,column sep=6em,minimum width=2em,
  text height=1.5ex, text depth=0.25ex]
  { \Par_{\vec{k}}(\vec{a})*\vec{r} & \Par_{\vec{m}}(\vec{a})*\vec{r} \\
    \Par_{\vec{k}}(\vec{b})*\vec{r} & \Par_{\vec{m}}(\vec{b})*\vec{r} \\};
  \path[-stealth,font=\scriptsize]
    (m-1-1) edge node [auto] {$\Par_{f}(\vec{a})\otimes^P\id$} (m-1-2)
            edge node [left] {$\Par_{\vec{k}}(\gamma)\otimes^N\id$} (m-2-1)
    (m-1-2) edge node [right] {$\Par_{\vec{m}}(\gamma)\otimes^N\id$} (m-2-2)
    (m-2-1) edge node [below] {$\Par_{f}(\vec{b})\otimes^P\id$} (m-2-2);
\end{tikzpicture}
\end{center}

\begin{center}
\begin{tikzpicture}
  \matrix (m) [matrix of math nodes,row sep=3em,column sep=6em,minimum width=2em,
  text height=1.5ex, text depth=0.25ex]
  { \Par_{\vec{k}*\vec{r}}(\vec{a}*(1)^{l(\vec{r})}) & \Par_{\vec{m}*\vec{r}}(\vec{a}*(1)^{l(\vec{r})}) \\
    \Par_{\vec{k}*\vec{r}}(\vec{b}*(1)^{l(\vec{r})}) & \Par_{\vec{m}*\vec{r}}(\vec{b}*(1)^{l(\vec{r})}) \\};
  \path[-stealth,font=\scriptsize]
    (m-1-1) edge node [auto] {$\Par_{f\otimes^P\id}(\vec{a}*(1)^{l(\vec{r})})$} (m-1-2)
            edge node [left] {$\Par_{\vec{k}*\vec{r}}(\gamma\otimes^N \id)$} (m-2-1)
    (m-1-2) edge node [right] {$\Par_{\vec{m}*\vec{r}}(\gamma\otimes^{N}\id)$} (m-2-2)
    (m-2-1) edge node [below] {$\Par_{f\otimes^P \id}(\vec{b}*(1)^{l(\vec{r})})$} (m-2-2);
\end{tikzpicture}
\end{center}

The same statement also holds for tensoring with $\id$ on the left.

In particular this implies that if $f,g\colon \vec{k}_0\rightarrow \vec{k}_1$ are morphisms in $Q_1(P)$ such that $f\sim_{\Par} g$, and $\vec{m}$ is a vector, then
$$(f\otimes_1 \id)\sim_{\Par} (g\otimes_1 \id)$$
and
$$(\id\otimes_1 f)\sim_{\Par} (\id\otimes_1 g),$$
which in turn implies that the monoidal product respects the partition relation. It follows that the symmetric monoidal structure on $Q_1(P)$ descends to a symmetric monoidal structure on $Q(P)$.

\bigskip $\bullet$ Let $P,P'$ be dg-props and let $g\colon P\rightarrow P'$ be a morphism of dg-props. The functor $g$ induces a symmetric monoidal functor $g^\otimes\colon P^\otimes\rightarrow P'^{\otimes}$. Define $Q_0(g)\colon Q_0(P)\rightarrow Q_0(P')$ to be the induced functor.It is the identity on objects and on generators coming from $\tilde{N}^\Sigma$, and acts by $g^\otimes$ on generators coming from $P^\otimes$. This functor induces a symmetric monoidal functor $Q(P)\rightarrow Q(P')$. To see this, observe that since $g$ is a prop morphism, we have $Q_0(g)(f\otimes_1 \id)=Q_0(g)(f)\otimes_1 \id$ for any morphism in $Q_0(P)$ and $Q(\Par_{f}(\vec{a}))=\Par_{g^\otimes(f)}(\vec{a})$ for any morphism in $P^\otimes$. This implies that $Q_0(g)$ descends to a functor $Q(g)\colon Q(P)\rightarrow Q(P')$. Symmetric monoidality of this functor can be seen by observing that
$$Q(g)(f_0)\otimes_1 Q(g)(f_1)=(\id \otimes_1 Q(g)(f_1))\circ (Q(g)(f_0)\otimes_1\id)$$
$$=  Q(g)(\id \otimes_1 f_1)\circ Q(g)(f_0\otimes_1 \id) = Q(g)\left((\id \otimes_1 f_1)\circ (f_0\otimes_1 \id)\right)=Q(g)(f_0\otimes_1 f_1)$$
and that since $Q(g)$ is the identity on $\tilde{N}^\Sigma$, the image of the twist morphism in $Q(P)$ is the twist morphism in $Q(P')$.

\end{construction}

\begin{lem}
The canonical functors $P^\otimes\rightarrow Q(P)$ and $\tilde{N}^\Sigma\rightarrow Q(P)$ are symmetric monoidal.
\end{lem}
\begin{proof}
Let $F\colon P^\otimes\rightarrow Q(P)$ be the canonical functor taking a morphism $f\colon \vec{k}\rightarrow \vec{m}$ to the morphism represented by the singleton word $f$. Note that for $g\colon \vec{m}\rightarrow \vec{l}$ another morphism in $P^\otimes$, we have
$$F(g)\circ F(f)\sim_{red} F(g\circ f)$$
such that this is indeed a functor. Since the twist morphism in $Q(P)$ lives in $\vec{Sigma}$, the functor $F$ takes the twist morphism of $P^\otimes$ to the twist morphism of $Q(P)$. To see that $F$ is symmetric monoidal, it is therefore sufficient to observe that for a pair of morphisms $f_i\colon \vec{k}_{0,i}\rightarrow \vec{m}_i$, $i=0,1$ in $P^\otimes$, we have
$$(\id\otimes^P f_1)\circ(f_0\otimes^P \id)\sim_{red} f_0\otimes^P f_1,$$
such that $F(f_0)\otimes_1 F(f_1)=F(f_0\otimes^P f_1)$ in $Q(P)$. The case for $\tilde{N}^\Sigma$ is identical.
\end{proof}

\begin{obs} \label{rmk-par-example}
The shuffle and Alexander-Whitney maps introduced in Notation \ref{notn-sh-aw} are special cases of the partition construction. In particular, if $\vec{k}$ has $l(\vec{k})=n$, then
$$\sh_{\vec{k}}=\Par_{\vec{k}}(\sh_n):\vec{k}\rightarrow (|\vec{k}|)$$
and similarly
$$AW_{\vec{k}}=\Par_{\vec{k}}(AW_n):(|\vec{k}|)\rightarrow \vec{k}.$$
The partition relations imply that if $\vec{k}_0$ and $\vec{k}_1$ are vectors with $l(\vec{k}_0)=l(\vec{k}_1)=n$  and $f=(f_1,...,f_n)\colon \vec{k}_0\rightarrow \vec{k}_1$ is a morphism in $P^\otimes$, then the following squares commute.
\begin{center}
\begin{tikzpicture}
\node (A1) at (-3,0) {$\vec{k}_0$};
\node (B1) at (0,0) {$\vec{k}_1$};
\node (C1) at (-3,-2) {$(|\vec{k}_0|)$};
\node (D1) at (0,-2) {$(|\vec{k}_1|)$};
\node (A2) at (3,0) {$(|\vec{k}_{0}|)$};
\node (B2) at (6,0) {$(|\vec{k}_{1}|)$};
\node (C2) at (3,-2) {$\vec{k}_{0}$};
\node (D2) at (6,-2) {$\vec{k}_{1}$};

\path[->,font=\scriptsize,>=angle 90]
(A1) edge node[above] {$f$} (B1)
(A1) edge node[left] {$\Par_{\vec{k}_0}(\sh_n)=\sh_{\vec{k}_{0}}$} (C1)
(B1) edge node[auto] {$\sh_{\vec{k}_{1}}$} (D1)
(C1) edge node[below] {$\Par_{f}((n))$} (D1);

\path[->,font=\scriptsize,>=angle 90]
(A2) edge node[above] {$\Par_{f}((n))$} (B2)
(A2) edge node[left] {$AW_{\vec{k}_{0}}$} (C2)
(B2) edge node[auto] {$AW_{\vec{k}_{1}}=\Par_{\vec{k}_1}(AW_n)$} (D2)
(C2) edge node[below] {$f$} (D2);
\end{tikzpicture}
\end{center}
\end{obs}

\begin{obs}
$\bullet$ For any $n$-tuple $f=(f_1,...,f_n)$ of morphisms in $P$, we have $\Par_f((1)^n)=f$.

$\bullet$ For a sequence
$$\vec{k}_0\xrightarrow{f}\vec{k}_1\xrightarrow{g}\vec{k}_2$$
of morphisms in $P^\otimes$, where the $l(\vec{k}_i)=n$ and $\vec{a}\in N_n$ we have $\Par_g(\vec{a})\circ \Par_f(\vec{a})=\Par_{g\circ f}(\vec{a})$.
\end{obs}

\begin{defn}
For $\C$ a category enriched in bicomplexes, let $\text{Tot}(\C)$ be the dg-category whose morphism complexes are the $\oplus$-totalization of the morphism bicomplexes in $\C$.
\end{defn}

\begin{lem} \label{defn-Q-projection}
There is a natural symmetric monoidal functor $F\colon \text{Tot}(Q(P))\rightarrow P$ defined on objects by taking $\vec{k}$ to $|\vec{k}|$, and on morphisms by taking $f=(f_1,...,f_n)\colon \vec{k}\rightarrow \vec{m}$ in $P^\otimes(\vec{k},\vec{m})$ to $\Par_{f}((n))\colon |\vec{k}|\rightarrow |\vec{m}|$ in $P$, and $\gamma\colon  \vec{k}\rightarrow \vec{m}$ in $\tilde{N}(\vec{k},\vec{m})_i$ to $\id_{|\vec{k}|}$ if $i=0$ and $0$ otherwise.
\end{lem}
\begin{proof}
It is clear that the assignment is natural in $P$ if it is well-defined, which we now verify. Given a morphism $\gamma\colon \vec{a}\rightarrow \vec{b}$ in $N_n$ and $f=(f_1,...,f_n)\colon \vec{k}\rightarrow \vec{m}$ in $P^{\otimes}$, we must verify that the diagram
\sq{\Par_{\vec{k}}(\vec{a}),\Par_{\vec{m}}(\vec{a}),\Par_{\vec{k}}(\vec{b}),\Par_{\vec{m}}(\vec{b}),\Par_{f}(\vec{a}),\Par_{\vec{m}}(\gamma),\Par_{\vec{k}}(\gamma),\Par_{f}(\vec{b})}
remains commutative after applying $F$. It is sufficient to assume that $\gamma\in \tilde{N}(\vec{a},\vec{b})_0$. But $F$ takes $\Par_{\vec{k}}(\gamma)$ to the identity and we have in general that
$$\Par_{\Par_f(\vec{a})}((l(\vec{a}))=\Par_f((n))$$
such that
$$F(\Par_f(\vec{a}))=\Par_f((n))=F(\Par_f(\vec{b})),$$
so $F$ is well defined. To see that $F$ preserves the differentials on each morphism complex, note that for a  morphism
$$f=\left\{\vec{k}_0\xrightarrow{\phi^0}\vec{k}_1\xrightarrow{\gamma^0}\vec{k}_2\xrightarrow{\phi^1}...\xrightarrow{\gamma^m} \vec{k}_{2m}\right\}$$
where $\phi^i\in P^{\otimes}(\vec{k}_{2i},\vec{k}_{2i+1})$ and $\gamma^i \in \tilde{N}(\vec{k}_{2i+1},\vec{k}_{2i+2})$, the differential is given by
$$df=(d_v \gamma^m)\circ \phi^m \circ ... \circ \gamma^0\circ \phi^0+ (-1)^{|\gamma^m|}\gamma^m\circ (d_h \phi^m) \circ ... \circ \gamma^0\circ \phi^0+...$$
$$+ (-1)^{|\gamma^m|+|\phi^m|+...+|\gamma^0|}\gamma^m \circ \phi^m \circ ... \circ \gamma^0\circ d_h\phi^0$$
In the case that $Ff$ is non-zero (i.e. each $|\gamma^i|=0$) this differential is identical to the differential in $P$.
\end{proof}

\begin{lem} \label{lem-Q-factorization}
Let $P$ be a dg-prop and let $\vec{k},\vec{m}\in\Ob Q(P)$. The map
$$\Hom_{\text{Tot}(Q(P))}(\vec{k},\vec{m})\rightarrow \Hom_P(|\vec{k}|,|\vec{m}|)$$
induced by $F$ is a quasi-isomorphism.
\end{lem}
\begin{proof}
Denote by $c_v\Hom_P(|\vec{k}|,|\vec{m}|)$ the bicomplex which has $\Hom_P(|\vec{k}|,|\vec{m}|)$ concentrated in vertical degree $0$, and consider the map of bicomplexes 
$$A\colon c_v\Hom_P(|\vec{k}|,|\vec{m}|)\simeq \tilde{N}((|\vec{m}|),(|\vec{m}|))\otimes \Hom_P(|\vec{k}|,|\vec{m}|)\rightarrow \Hom_{Q(P)}(\vec{k},\vec{m})$$
where the left map is the homotopy equivalence taking $f\in \Hom_P(|\vec{k}|,|\vec{m}|)$ to $\id\otimes f$, and the right map takes $\gamma\otimes f$ to $AW_{\vec{m}}\circ \gamma \circ f \circ \sh_{\vec{k}}$. We will show that the totalization of $A$ is a quasi-isomorphism and a quasi-inverse to the map induced by $F$ on Hom-complexes.

Let $f\colon \vec{k}\rightarrow \vec{m}$ with $|f|=(d,d')$ in $Q(P)$ be a composition of generators of $Q(P)$, i.e. $f$ is represented by an elementary tensor 
$$\vec{k}_0\xrightarrow{\phi^0}\vec{k}_1 \xrightarrow{\gamma^0} \vec{k}_2 \xrightarrow{\phi^1} ... \xrightarrow{\gamma^{m-1}} \vec{k}_{2m}$$
in $Q_0(P)(\vec{k},\vec{m})$, where $\phi^i\in P^{\otimes}(\vec{k}_{2i},\vec{k}_{2i+1})$ and $\gamma^i \in \tilde{N}(\vec{k}_{2i+1},\vec{k}_{2i+2})_0$. We write
$$n_i:=l(\vec{k}_{2i})=l(\vec{k}_{2i+1}).$$
For such an $f$, we write $G(f)\in\Hom_P(|\vec{k}_0|,|\vec{k}_{2m}|)$ for the morphism
$$G(f)=\Par_{\phi^{m-1}}((n_{m-1}))\circ ... \circ \Par_{\phi^0}((n_0)).$$
Note that $G(f)$ does not depend on the representative of $f$. In particular, well-definedness with respect to $\sim_{\Par}$ can be seen by the identity
$$\Par_{\Par_{\phi}(\vec{a})}((l(\vec{a}))=\Par_\phi ((n)).$$
for any $\phi\colon \vec{k}\rightarrow \vec{m}$ in $P^\otimes$ with $l(\vec{k})=n$.

If for any $\phi$ which is represented by an elementary tensor, the homology class of $\phi$ is represented by a composition $AW_{\vec{m}}\circ \tilde{\gamma} \circ (G(\phi)) \circ \sh_{\vec{k}}$, where $\tilde{\gamma}\in \tilde{N}((|\vec{m}|),(|\vec{m}|))_{d'}$, then the map $A$ above is a quasi-isomorphism after totalizing. Indeed, assume that $f\in \Hom_{Q(P)}(\vec{k},\vec{m})$ is a cycle with respect to the vertical differential. It is given by a sum
$$f=f_1+...+f_n\in \Hom_{Q(P)}(\vec{k},\vec{m})_{d,n}.$$
where each $f_i$ is represented by an elementary tensor. We may assume that each $f_i$ has the form $AW_{\vec{m}}\circ \gamma_i \circ (G(f_i)) \circ \sh_{\vec{k}}$. Using the contractibility of $\tilde{N}$, we may in fact assume that the $\gamma_i$ are identical, such that $f$ represents the same vertical homology class as
$$AW_{\vec{m}}\circ \gamma \circ (G(f)) \circ \sh_{\vec{k}}$$
for some $\gamma\in \tilde{N}((|\vec{m}|),(|\vec{m}|))$. Now the vertical differential acts only on $\gamma$, which must be a cycle, hence a boundary in $\tilde{N}((|\vec{m}|),(|\vec{m}|))$ unless $d=0$, hence this cycle represents a trivial homology class if $d>0$. In the case $d=0$, $\gamma=\id$ is a cycle which is not a boundary. It follows that on homology,
$$H_*(\Hom_{Q(P)}(\vec{k},\vec{m});d_v)\simeq \Hom_P(|\vec{k}|,|\vec{m}|).$$
We get an isomorphism of $E_1$-pages of the spectral sequence for a double complex:
$$H_*(\tilde{N}((|\vec{m}|),(|\vec{m}|))\otimes \Hom_P(|\vec{k}|,|\vec{m}|),d_v)\xrightarrow{\sim} H_*(\Hom_{Q(P)}(\vec{k},\vec{m});d_v)$$
hence $A$ is a quasi-isomorphism after totalizing. Now, for any $f\in \Hom_P(k,m)$ we have
$$F\circ \Tot(A)(f)=F(AW_{\vec{m}} \circ (f) \circ \sh_{\vec{k}})=f$$
such that $F\circ \Tot(A)$ is the identity. By the 2-out-of-3 property for quasi-isomorphisms, $F$ induced quasi-isomorphisms on Hom-complexes.

In the following, for $a,b$ elements of a bicomplex $C$ with $|a|=|b|=(d,d')$, a \emph{vertical homotopy} $h\colon  a\simeq b$ means an element $h$ of $C$ with $|h|=(d,d'+1)$ such that $d_vh=b-a$.

We are left to show that each morphism $f\in\Hom_{Q(P)}(\vec{k},\vec{m})$ which is a composition of generators admits a homotopy to the desired form. This is done by performing (strong) induction on the vertical degree $d'$. For $d'=0$, the morphism $f$ is given \emph{a priori} by a sequence of generators
$$\vec{k}_0\xrightarrow{\phi^0}\vec{k}_1 \xrightarrow{\gamma^0} \vec{k}_2 \xrightarrow{\phi^1} ... \xrightarrow{\gamma^{m-1}} \vec{k}_{2m}$$
where  $\phi^i\in P^{\otimes}(\vec{k}_{2i},\vec{k}_{2i+1})$ and $\gamma^i \in \tilde{N}(\vec{k}_{2i+1},\vec{k}_{2i+2})_0$. Write $n_i:=l(\vec{k}_{2i})=l(\vec{k}_{2i+1})$.
To begin, we may fix for each $\gamma^i$ a vertical homotopy $c(\gamma^i)\colon  \gamma^i \simeq AW_{\vec{k}_{2i+2}}\circ \sh_{\vec{k}_{2i+1}}$. Applying the $c(\gamma^{i})$ we obtain a new morphism
$$g=\vec{k}_0\xrightarrow{\phi^0}\vec{k}_1 \xrightarrow{AW_{\vec{k}_{2}}\circ \sh_{\vec{k}_1}} \vec{k}_{2} \xrightarrow{\phi^1} ... \xrightarrow{AW_{\vec{k}_{2m}}\circ \sh_{\vec{k}_{2m-1}}} \vec{k}_{2m}$$
equipped with a vertical homotopy $f\simeq g$.
Now repeated application of the relations 
\begin{center}
\begin{tikzpicture}
\node (A1) at (-3,0) {$\vec{k}_{2i}$};
\node (B1) at (0,0) {$\vec{k}_{2i+1}$};
\node (C1) at (-3,-2) {$(|\vec{k}_{2i}|)$};
\node (D1) at (0,-2) {$(|\vec{k}_{2i+1}|)$};
\node (A2) at (3,0) {$(|\vec{k}_{2i}|)$};
\node (B2) at (6,0) {$(|\vec{k}_{2i+1}|)$};
\node (C2) at (3,-2) {$\vec{k}_{2i}$};
\node (D2) at (6,-2) {$\vec{k}_{2i+1}$};

\path[->,font=\scriptsize,>=angle 90]
(A1) edge node[above] {$\phi^i$} (B1)
(A1) edge node[left] {$\sh_{\vec{k}_{2i}}$} (C1)
(B1) edge node[auto] {$\sh_{\vec{k}_{2i+1}}$} (D1)
(C1) edge node[below] {$\Par_{\phi^i}((n_i))$} (D1);

\path[->,font=\scriptsize,>=angle 90]
(A2) edge node[above] {$\Par_{\phi^i}((n_i))$} (B2)
(A2) edge node[left] {$AW_{\vec{k}_{2i}}$} (C2)
(B2) edge node[auto] {$AW_{\vec{k}_{2i+1}}$} (D2)
(C2) edge node[below] {$\phi^i$} (D2);
\end{tikzpicture}
\end{center}
allows us to rewrite $g$ as the composition
$$g= AW_{\vec{k}_{2m}}\circ \Par_{\phi^{m-1}}((n_{m-1}))\circ (AW_{\vec{k}_{2m-2}}\circ \sh_{\vec{k}_{2m-2}}) \circ ... $$
$$ ... \circ \Par_{\phi^1}((n_1)) \circ (AW_{\vec{k}_{2}}\circ \sh_{\vec{k}_{2}})\circ \Par_{\phi^0}((n_0)) \circ \sh_{\vec{k}_0}$$
Now choose vertical homotopies $\beta_{\vec{k}_{2i}}\colon (AW_{\vec{k}_{2i}}\circ \sh_{\vec{k}_{2i}})\rightarrow \id_{(|\vec{k}_{2i}|)}$, giving us a composition of the desired form. This completes the base case.

For $d'>0$, we may again write $f$ as a sequence of generators
$$\vec{k}_0\xrightarrow{\phi^0}\vec{k}_1 \xrightarrow{\gamma^0} \vec{k}_{2} \xrightarrow{\phi^1} ... \xrightarrow{\gamma^{m-1}} \vec{k}_{2m}$$
where $\phi^i\in P^{\otimes}(\vec{k}_{2i},\vec{k}_{2i+1})$ with $n_i:=l(\vec{k}_{2i})=l(\vec{k}_{2i+1})$ and now $\gamma^i \in \tilde{N}(\vec{k}_{2i+1},\vec{k}_{i+1,0})_{d'_i}$. We now consider two cases. Assume first that $d'_i<d'$ for each $i$. Let $j$ be the least $i$ such that $|\gamma^i|>0$. By our assumption on the $d'_i$, $j<m-1$. Write $f'$ for the composition 
$$\vec{k}_{2j+2}\xrightarrow{\phi^{j+1}}\vec{k}_{2j+3} \xrightarrow{\gamma^{j+1}} \vec{k}_{j+2,0} \xrightarrow{\phi^{j+2}} ... \xrightarrow{\gamma^{m-1}} \vec{k}_{2m}$$
and write $f''$ for the composition
$$\vec{k}_0\xrightarrow{\phi^0}\vec{k}_1 \xrightarrow{\gamma^0} \vec{k}_{2} \xrightarrow{\phi^1} ... \xrightarrow{\gamma^{j}} \vec{k}_{2j+2}.$$
By induction, we may rewrite $f'$ and $f''$ up to homotopy as
$$f'\simeq AW_{\vec{k}_{2m}}\circ \tilde{\gamma}' \circ (G(f')) \circ \sh_{\vec{k}_{2j+2}}$$
where $\tilde{\gamma}'\in \tilde{N}((m),(m))_{d'_{j+1}+...+d'_{m-1}}$, and
$$f''\simeq AW_{\vec{k}_{2j+2}}\circ \tilde{\gamma}'' \circ (G(f'')) \circ \sh_{\vec{k}_0}$$
where $\tilde{\gamma}''\in \tilde{N}((|\vec{k}_{2j+2}|),(|\vec{k}_{2j+2}|))_{d'_{0}+...+d'_{j}}$. Hence we get a vertical homotopy
$$f\simeq AW_{\vec{k}_{2m}}\circ \tilde{\gamma}' \circ (G(f'))\circ \tilde{\gamma}'' \circ (G(f'')) \circ \sh_{\vec{k}_0}$$
By induction, we may now rewrite
$$(G(f'))\circ \tilde{\gamma}'' \circ(G(f''))\simeq \tilde{\gamma}''' \circ (G(f'))\circ (G(f''))=\tilde{\gamma}''' \circ (G(f))$$
for some $\tilde{\gamma}'''\in \tilde{N}((|\vec{k}_{2m}|),(|\vec{k}_{2m}|))_{d'_{0}+...+d'_{j}}$ to obtain a composition of the desired form. This completes the case when $d_i'<d'$ for all $i$.

Finally, assume that there is a $j$ such that $d'_j=d'$. If $j=m-1$, then the result follows from the base case and contractibility of $\tilde{N}((|\vec{m}|),(|\vec{m}|))$. Namely, for a morphism $f$ in $Q(P)$ represented by a sequence of generators
$$\vec{k}_0\xrightarrow{\phi^0}\vec{k}_1 \xrightarrow{\gamma^0} \vec{k}_{2} \xrightarrow{\phi^1} ... \xrightarrow{\gamma^{m-1}} \vec{k}_{2m}$$
as above, where $|\gamma^{m-1}|=d'$, we get vertical homotopies
$$f\simeq \gamma^{m-1}\circ AW_{\vec{k}_{2m-1}}\circ \gamma' \circ (G(f))\circ \sh_{\vec{k}_0}$$
$$\simeq AW_{\vec{k}_{2m}} \circ \gamma \circ (G(f))\circ \sh_{\vec{k}_0}$$
for some $\gamma'\in \tilde{N}((|\vec{k}_{2m-1}|),(|\vec{k}_{2m-1}|))_{0}$ and $\gamma\in \tilde{N}((|\vec{k}_{2m}|),(|\vec{k}_{2m}|))_{d'}$, where the first homotopy comes from the base case, and the second by contractibility of $\tilde{N}((|\vec{m}|),(|\vec{m}|))$.

If $j<m-1$, we will provide a vertical homotopy between $f$ and another morphism $f'$ for which $d'_{j+1}=d'$. By the above, this will finish the argument. We apply a homotopy $\gamma^j\simeq AW_{\vec{k}_{2j+2}}\circ \bar{\gamma}^j \circ \sh_{\vec{k}_{2j+1}}$ where $\bar{\gamma}^j\in \tilde{N}((|\vec{k}_{2j+2}|),(|\vec{k}_{2j+2}|))_{d'}$. By contractibility, we may assume that $\bar{\gamma}_j$ is of the form $\Par_{\vec{k}_{2j+2}}(\gamma'_j)$ for some $\gamma'_j\in \tilde{N}((n_{j+1}),(n_{j+1}))_{d'}$. To name a concrete such element, one can use the (higher) homotopies witnessing $\sh$ and $AW$ as mutual homotopy inverses. Now using the relations
\begin{center}
\begin{tikzpicture}
\node (A) at (-5,0) {$(|\vec{k}_{2j+2}|)$};
\node (B) at (0,0) {$(|\vec{k}_{2j+3}|)$};
\node (C) at (-5,-2) {$(|\vec{k}_{2j+2}|)$};
\node (D) at (0,-2) {$(|\vec{k}_{2j+3}|)$};

\path[->,font=\scriptsize,>=angle 90]
(A) edge node[above] {$\Par_{\phi^{j+1}}((n_j))$} (B)
(A) edge node[left] {$\Par_{\vec{k}_{2j+2}}(\gamma'_j)$} (C)
(B) edge node[auto] {$\Par_{\vec{k}_{2j+3}}(\gamma'_j)$} (D)
(C) edge node[below] {$\Par_{\phi^{j+1}}((n_j))$} (D);
\end{tikzpicture}
\end{center}
we obtain a composition
$$f\simeq f'= \left\{ \vec{k}'_{0}\xrightarrow{\phi'^0}\vec{k}'_{1} \xrightarrow{\gamma'^0} \vec{k}'_{2} \xrightarrow{\phi'^1} ... \xrightarrow{\gamma'^{m-1}} \vec{k}'_{2m} \right\}$$
where for $i\neq j+1$ we have $\vec{k}'_{2i}=\vec{k}_{2i}$ and $\vec{k}'_{2i+1}=\vec{k}_{2i+1}$, for $i\neq j+1$ we have $\phi'^i=\phi^i$, and for $i\neq j+1,j$ we have $\gamma^i=\gamma'^i$. Finally, $\vec{k}'_{2j+2}=(|\vec{k}_{2j}|)$, $\vec{k}'_{2j+3}=(|\vec{k}_{2j+3}|)$, $\phi'^{j+1}=\Par_{\phi^{j+1}}((n_{j+1}))$, $\gamma'^j=\sh_{\vec{k}_{2j+1}}$ and $\gamma'^{j+1}=\gamma^{j+1}\circ AW_{\vec{k}_{2j+3}}\circ \Par_{\vec{k}_{2j+3}}(\gamma'_j)$. Collecting the differences between $f$ and $f'$ in a diagram, we have
\begin{center}
\begin{tikzpicture}
\node (A1) at (0,0) {$\vec{k}_{2j+1}$};
\node (A2) at (4,0) {$\vec{k}_{2j+2}$};
\node (A3) at (8,0) {$\vec{k}_{2j+3}$};
\node (A4) at (12,0) {$\vec{k}_{2j+4}$};
\node (B1) at (1,-2) {$(|\vec{k}_{2j+1}|)$};
\node (B2) at (5,-2) {$(|\vec{k}_{2j+1}|)$};
\node (C1) at (2,-4) {$(|\vec{k}_{2j+3}|)$};
\node (C2) at (6,-4) {$(|\vec{k}_{2j+3}|)$};

\node (S) at (2,-1) {$(\sim)$};

\path[->,font=\scriptsize,>=angle 90]
(A1) edge node[auto] {$\gamma^j$} (A2)
	 edge node[left] {$\gamma'^j=\sh_{\vec{k}_{2j+1}}$} (B1)
(A2) edge node[auto] {$\phi^{j+1}$} (A3)
(A3) edge node[auto] {$\gamma^{j+1}$} (A4)
(B1) edge node[auto] {$\Par_{\vec{k}_{2j+2}}(\gamma'_j)$} (B2)
	 edge node[left] {$\phi'^{j+1}=\Par_{\phi^{j+1}}((n_{j+1}))$} (C1)
(B2) edge node[left] {$AW_{\vec{k}_{2j+2}}$} (A2)
	 edge node[left] {$\Par_{\phi^{j+1}}((n_{j+1}))$} (C2)
(C1) edge node[auto] {$\Par_{\vec{k}_{2j+2}}(\gamma'_j)$} (C2)
(C2) edge node[right] {$AW_{\vec{k}_{2j+3}}$} (A3)
(C1) edge [bend right=40] node[below right] {$\gamma'^{j+1}$} (A4);
\end{tikzpicture}
\end{center}
where the square marked $(\sim)$ commutes up to vertical homotopy.

We see now that for the composition $f'$, $|\gamma'^{j+1}|=d'$, and this finishes the argument.
\end{proof}

Recall that for a prop $P$, a $P$-algebra is a symmetric monoidal functor $\Phi\colon P\rightarrow \Ch_k$ and a $\Ass \otimes P$-algebra is the same as a symmetric monoidal functor $P\rightarrow \dgAlg_k$.

\begin{lem} \label{lem-alpha-TotQ}
The functor $\text{Tot}(Q(-))\colon \mathsf{dgprop}\rightarrow \dgCat^{\otimes}$ has the property that there is a natural transformation of functors $\mathsf{dgprop}^{op}\rightarrow\Cat$
$$\alpha\colon \Fun^{\otimes}(-,\dgAlg_k)\rightarrow \Fun^{\otimes}(\text{Tot}(Q(-)),\Ch_k)$$
such that $\alpha_P(\Phi)(1)=C(\Phi(1))$.
\end{lem}
\begin{proof}
We divide the proof into several steps. First we construct the functors $\alpha_P(\Phi)$. Then we show functoriality in $\Phi$. Finally we will show naturality in $P$.

\bigskip \underline{Step 1: Constructing $\alpha_P(\Phi)$.}

\bigskip Let $P$ be a dg-prop, and let $\Phi\colon P\rightarrow \dgAlg_k$ be a symmetric monoidal functor. Throughout this section of the proof, we write $A=\Phi(1)$ for ease of notation. We will produce a symmetric monoidal functor $\alpha_P(\Phi)\colon  \text{Tot}(Q(P))\rightarrow \Ch_k$ that sends $\vec{k}$ to $C^{\vec{k}}(A)$ (recall Notation \ref{notn-C-vec}). We describe the action of $\alpha_P(\Phi)$ on morphisms in terms of the generators of $\text{Tot}(Q(P))$. If $f\colon n\rightarrow m$ is in $P$, then $(f)$ acts by
$$C^{(n)}(A)\simeq C(\Phi(n))\xrightarrow{C(f)} C(\Phi(m))\simeq C^{(m)}(A).$$
This determines the action of morphisms in $P^\otimes$. Furthermore, $\tilde{N}$ acts according to (the proof of) Proposition \ref{defn-action-hochschild}. This determines how the generators of $\text{Tot}(Q(P))$ act. Note that the relations $\sim_{red}$ are preserved by this assignment. Furthermore, if $f\colon \vec{k}_0\rightarrow \vec{k}_1$ is in $P^\otimes$ and $\psi\colon \vec{m}_0\rightarrow \vec{m}_1$ is in $\tilde{N}^\Sigma$, we have isomorphisms $C^{\vec{k_0}*\vec{m}_0}(A)= C^{\vec{k}_0}(A)\otimes C^{\vec{m}_0}(A)$, and modulo these isomorphisms, we have
$$\alpha_P(\Phi)(\id_{\vec{m}_0}\otimes^P f)=\id_{C^{\vec{m_0}}(A)}\otimes \alpha_P(\Phi)(f)$$
and
$$\alpha_P(\Phi)(\phi\otimes^N \id_{\vec{k}_0})=\alpha_P(\Phi)(\psi) \otimes \id_{C^{\vec{k_0}}(A)}$$
such that the interchange relation is preserved. Since
$$\left(\id_{C^{\vec{m_0}}(A)}\otimes \alpha_P(\Phi)(f)\right) \circ \left( \alpha_P(\Phi)(\psi) \otimes \id_{C^{\vec{k_0}}(A)} \right)= \left(\alpha_P(\Phi)(f) \otimes \alpha_P(\Phi)(\psi) \right)$$
we have
$$\alpha_P(\Phi)(f\otimes_1 g)= \alpha_P(\Phi)(f)\otimes \alpha_P(\Phi)(g)$$
such that $\alpha_P(\Phi)$ defines a symmetric monoidal functor $\text{Tot}(Q_1(P))\rightarrow \Ch_k$.

To see that this assignment descends to a symmetric monoidal functor $\text{Tot}(Q(P))\rightarrow \Ch_k$, we are left to verify that the partition relations are preserved. Let $f=(f_1,...,f_n)\colon \vec{k}\rightarrow \vec{m}$ be an $n$-tuple of morphisms in $P$ and let $\gamma\colon \vec{a}\rightarrow \vec{b}$ be a morphism in $N_n$. We are will verify that the relations
\sq{\Par_{\vec{k}}(\vec{a}),\Par_{\vec{m}}(\vec{a}),\Par_{\vec{k}}(\vec{b}),\Par_{\vec{m}}(\vec{b}),\Par_{f}(\vec{a}),\Par_{\vec{m}}(\gamma),\Par_{\vec{k}}(\gamma),\Par_{f}(\vec{b})}
are preserved under the assignment $\alpha_P(\Phi)$. There is an isomorphism (see Definition \ref{defn-theta-iso} and the proof of Proposition \ref{defn-action-hochschild})
$$\alpha_P(\Phi)(\Par_{\vec{k}}(\vec{a}))=C^{\Par_{\vec{k}}(\vec{a})}(A) \stackrel{N\theta_{\vec{a}}^{-1}}{\rightarrow} N^{\vec{a}}(B^{cy}(A^{\otimes k_1}),...,B^{cy}(A^{\otimes k_n})).$$
Write $N^{\vec{a}}(B^{cy}(A^{\vec{k}}))$ for the latter. Consider the following diagrams:
\begin{center}
\begin{tikzpicture}
\node (A1) at (-4,0) {$N^{\vec{a}}(B^{cy}(A^{\vec{k}}))$};
\node (B1) at (0,0) {$N^{\vec{a}}(B^{cy}(A^{\vec{m}}))$};
\node (C1) at (-4,-2) {$N^{\vec{b}}(B^{cy}(A^{\vec{k}}))$};
\node (D1) at (0,-2) {$N^{\vec{b}}(B^{cy}(A^{\vec{m}}))$};
\node (A2) at (3,0) {$C^{\Par_{\vec{k}}(\vec{a})}(A)$};
\node (B2) at (7,0) {$C^{\Par_{\vec{m}}(\vec{a})}(A)$};
\node (C2) at (3,-2) {$N^{\vec{a}}(B^{cy}(A^{\vec{k}}))$};
\node (D2) at (7,-2) {$N^{\vec{a}}(B^{cy}(A^{\vec{m}}))$};

\path[->,font=\scriptsize,>=angle 90]

(A1) edge node[above] {$N^{\vec{a}}(B^{cy}(f))$} (B1)
(B1) edge node[auto] {$\gamma_{B^{cy}(A^{\vec{m}})}$} (D1)
(A1) edge node[left] {$\gamma_{B^{cy}(A^{\vec{k}})}$} (C1)
(C1) edge node[below] {$N^{\vec{b}}(B^{cy}(f))$} (D1);

\path[->,font=\scriptsize,>=angle 90]
(A2) edge node[above] {$\alpha_P(\Phi)(\Par_{f}(\vec{a}))$} (B2)
(B2) edge node[auto] {$N\theta_{\vec{a}}^{-1}$} (D2)
(A2) edge node[left] {$N\theta_{\vec{a}}^{-1}$} (C2)
(C2) edge node[below] {$N^{\vec{a}}(B^{cy}(f))$} (D2);
\end{tikzpicture}
\end{center}
The left diagram commutes by the definition of $\tilde{N}$, while the right diagram commutes by the naturality of the symmetric monoidal structure maps of $B^{cy}$. Finally, observe that
$$\alpha_P(\Phi)(\Par_{\vec{k}}(\gamma))=N\theta_{\vec{b}}\circ \gamma_{B^{cy}(A^{\vec{k}})}\circ N\theta^{-1}_{\vec{a}}$$
Together these facts imply that the relations in $\text{Tot}(Q(P))$ are preserved by the action. To be precise, we have a commutative diagram
\begin{center}
\begin{tikzpicture}

\node (A1) at (0,0) {$C^{\Par_{\vec{k}}(\vec{a})}(A)$};
\node (A2) at (4,0) {$C^{\Par_{\vec{m}}(\vec{a})}(A)$};
\node (B1) at (0,-2) {$N^{\vec{a}}(B^{cy}(A^{\vec{k}}))$};
\node (B2) at (4,-2) {$N^{\vec{a}}(B^{cy}(A^{\vec{m}}))$};
\node (C1) at (0,-4) {$N^{\vec{b}}(B^{cy}(A^{\vec{k}}))$};
\node (C2) at (4,-4) {$N^{\vec{b}}(B^{cy}(A^{\vec{m}}))$};
\node (D1) at (0,-6) {$C^{\Par_{\vec{k}}(\vec{b})}(A)$};
\node (D2) at (4,-6) {$C^{\Par_{\vec{m}}(\vec{b})}(A)$};

\path[->,font=\scriptsize,>=angle 90]

(A1) edge node[auto] {$\alpha_P(\Phi)(\Par_{f}(\vec{a}))$} (A2)
	 edge node[auto] {$N\theta_{\vec{a}}^{-1}$} (B1)
(A2) edge node[auto] {$N\theta_{\vec{a}}^{-1}$} (B2)
(B1) edge node[auto] {$N^{\vec{a}}(B^{cy}(f))$} (B2)
	 edge node[auto] {$\gamma_{B^{cy}(A^{\vec{k}})}$} (C1)
(B2) edge node[auto] {$\gamma_{B^{cy}(A^{\vec{m}})}$} (C2)
(C1) edge node[auto] {$N^{\vec{b}}(B^{cy}(f))$} (C2)
	 edge node[auto] {$N\theta_{\vec{b}}$} (D1)
(C2) edge node[auto] {$N\theta_{\vec{b}}$} (D2)
(D1) edge node[auto] {$\alpha_P(\Phi)(\Par_{f}(\vec{b}))$} (D2)
(A1.180) edge [bend right=30] node[left] {$\alpha_P(\Phi)(\Par_{\vec{k}}(\gamma)$} (D1.180)
(A2.0) edge [bend left=30] node[right] {$\alpha_P(\Phi)(\Par_{\vec{m}}(\gamma)$} (D2.0)
;
\end{tikzpicture}
\end{center}

Thus $\alpha_P(\Phi)$ is a functor from $\text{Tot}(Q(P))$ as claimed. 

\bigskip \underline{Step 2: Showing that $\alpha_P$ is a functor.}

\bigskip We will notationally identify an object in $\Fun^{\otimes}(\Ass\otimes P,\Ch_k)$ with its value at $1$. Let $\phi\colon A\rightarrow B$ be a morphism in $\Fun^{\otimes}(P,\dgAlg_k)$. We will produce a natural transformation $\alpha_P(A)\rightarrow \alpha_P(B)$. The component at $\vec{l}\in Q(P)$ is given by applying $\phi_1\colon A\rightarrow B$ componentwise, i.e.~$\alpha_P(\phi)_{\vec{l}}= C^{\vec{l}}(\phi_1) \colon C^{\vec{l}}(A)\rightarrow C^{\vec{l}}(B)$. It is sufficient to check naturality against the generators of $Q(P)$. If $f\colon n\rightarrow m$ is a morphism in $P$, then the following diagram commutes because it commutes before applying $C(-)$.
\sq{C(A^{\otimes n}),C(B^{\otimes n}),C(A^{\otimes m}),C(B^{\otimes m}),C(\phi_n),C(f_B),C(f_A),C(\phi_m)}
Let $\gamma\colon \vec{k}\rightarrow\vec{m}$ be a morphism in $\tilde{N}$ and consider the following diagrams:
\begin{center}
\begin{tikzpicture}
\node (A1) at (-4,0) {$C^{\vec{k}}(A)$};
\node (B1) at (-1,0) {$N^{\vec{k}}(B^{cy}(A))$};
\node (C1) at (-4,-2) {$C^{\vec{k}}(B)$};
\node (D1) at (-1,-2) {$N^{\vec{k}}(B^{cy}(B))$};
\node (A2) at (3,0) {$N^{\vec{k}}(B^{cy}(A))$};
\node (B2) at (7,0) {$N^{\vec{m}}(B^{cy}(A))$};
\node (C2) at (3,-2) {$N^{\vec{k}}(B^{cy}(B))$};
\node (D2) at (7,-2) {$N^{\vec{m}}(B^{cy}(B))$};

\path[->,font=\scriptsize,>=angle 90]

(A1) edge node[above] {$N\theta^{-1}_{\vec{k}}$} (B1)
(B1) edge node[auto] {$N^{\vec{k}}(B^{cy}(\phi))$} (D1)
(A1) edge node[left] {$C^{\vec{k}}(\phi)$} (C1)
(C1) edge node[below] {$N\theta^{-1}_{\vec{k}}$} (D1);

\path[->,font=\scriptsize,>=angle 90]
(A2) edge node[above] {$N^\gamma (B^{cy}(A))$} (B2)
(B2) edge node[auto] {$N^{\vec{m}}(B^{cy}(\phi))$} (D2)
(A2) edge node[left] {$N^{\vec{k}}(B^{cy}(\phi))$} (C2)
(C2) edge node[below] {$N^{\gamma}(B^{cy}(B))$} (D2);
\end{tikzpicture}
\end{center}
The left diagram commutes by naturality of the symmetric monoidal structure maps of $B^{cy}$ and the right diagram commutes by the definition of $\tilde{N}$. Since $\gamma$ acts by
$$\alpha_P(A)(\gamma)=N\theta_{m}\circ N^\gamma (B^{cy}(A)) \circ N\theta^{-1}_{\vec{k}}$$
the commutativity of these two families of diagrams implies naturality with respect to the morphisms in $\tilde{N}^\Sigma$.

\bigskip \underline{Step 3: Showing that $\alpha$ is natural in $P$.}

\bigskip Let $i\colon P\rightarrow P'$ be a morphism of dg-props. We need to check commutativity of the diagram
\begin{center}
\begin{tikzpicture}
\node (A) at (-6,0) {$\Fun^{\otimes}(P',\dgAlg_k)$};
\node (B) at (0,0) {$\Fun^{\otimes}(\text{Tot}(Q(P')),\Ch_k)$};
\node (C) at (-6,-2) {$\Fun^{\otimes}(P,\dgAlg_k)$};
\node (D) at (0,-2) {$\Fun^{\otimes}(\text{Tot}(Q(P)),\Ch_k)$};

\path[->,font=\scriptsize,>=angle 90]
(A) edge node[above] {$\alpha_{P'}$} (B)
(A) edge node[left] {$i^*$} (C)
(B) edge node[auto] {$\text{Tot}(Q(i))^*$} (D)
(C) edge node[below] {$\alpha_{P}$} (D);
\end{tikzpicture}
\end{center}
Let $\Phi\colon P'\rightarrow \dgAlg_k$ be a symmetric monoidal functor. We first show that the functors $\alpha_{P}(i^*\Phi)$ and $\text{Tot}(Q(i))^* \alpha_{P'}(\Phi)$ are equal. Since $i$ and $\text{Tot}(Q(i))$ are isomorphisms on objects, we have
$$\alpha_P(i^*\Phi)(\vec{k})=C^{\vec{k}}(i^*\Phi(1))=C^{\vec{k}}(\Phi(1))$$
and
$$\text{Tot}(Q(i))^*\alpha_{P'}(\Phi)(\vec{k})= \alpha_{P'}(\Phi)(\vec{k})=C^{\vec{k}}(\Phi(1))$$
so they are equal on objects. Let $\gamma\colon \vec{k}\rightarrow \vec{m}$ be a morphism in $\tilde{N}^{\Sigma}$. Since $\text{Tot}(Q(i))$ is the identity on $\tilde{N}^{\Sigma}$, we similarly have
$$\text{Tot}(Q(i))^*\alpha_{P'}(\Phi)(\gamma)=\alpha_{P'}(\Phi)(\gamma)= N\theta_{m}\circ N^\gamma (B^{cy}(\Phi(1))) \circ N\theta^{-1}_{\vec{k}}$$
and
$$\alpha_P(i^*\Phi)(\gamma)=N\theta_{m}\circ N^\gamma (B^{cy}(i^*\Phi(1))) \circ N\theta^{-1}_{\vec{k}}=N\theta_{m}\circ N^\gamma (B^{cy}(\Phi(1))) \circ N\theta^{-1}_{\vec{k}}$$
so the action of $\tilde{N}^{\Sigma}$ coincides as well. We now compare the action by a morphism $f\colon k\rightarrow m$ in $P$. We have
$$\alpha_P(i^*\Phi)(f)\colon  C^{(k)}(\Phi(1))\simeq C(\Phi(k))\xrightarrow{C(i(f))}C(\Phi(m))\simeq C^{(m)}(\Phi(1)).$$
Notice that $\alpha_P(i^*\Phi)(f)=\alpha_{P'}(\Phi)(i(f))$. Now since $\text{Tot}(Q(i))(f)=i(f)$ in $\text{Tot}(Q(P'))$, we have
$$\text{Tot}(Q(i))^*\alpha_{P'}(\Phi)(f)=\alpha_{P'}(\Phi)(i(f))$$
so the two functors coincide on objects.

Before we verify that the functors also agree on morphisms, we recall a basic fact about compositions of natural transformations. If $\C,\C',\D$ are categories, $j\colon \C\rightarrow C'$ is a functor and $\alpha\colon F\Rightarrow G \colon  \C'\rightarrow \D$ is a natural transformation, then the pullback of $\alpha$ along $j$ is given componentwise by $(\alpha* \id_j)_c = \alpha_{j(c)}$.

For a morphism $\psi\colon \Phi\rightarrow \Psi$, we have the natural transformations
$$\alpha_P(i^*\psi)\colon \alpha_P(i^*\Phi)\rightarrow \alpha_P(i^* \Psi)$$
and
$$\text{Tot}(Q(i))^*\alpha_{P'}(\psi)\colon \text{Tot}(Q(i))^*\alpha_{P'}(\Phi)\rightarrow \text{Tot}(Q(i))^*\alpha_{P'}(\Psi)$$
of functors $\text{Tot}(Q(P))\rightarrow \Ch_k$. It is sufficient to check that they coincide on components. Let $\vec{k}$ be an object of $\text{Tot}(Q(P))$. Then since $i$ is the identity on objects, we get

$$\alpha_P(i^*\psi)(\vec{k})=C^{\vec{k}}(i^*\psi(1))=C^{\vec{k}}(\psi(1))$$
and
$$\text{Tot}(Q(i))^*\alpha_{P'}(\psi)(\vec{k})= \alpha_{P'}(\psi)(\vec{k})=C^{\vec{k}}(\psi(1))$$
so they are equal.
\end{proof}

\bigskip\noindent \emph{Proof of Theorem A}: Define $\tilde{(-)}\colon \mathsf{dgprop}\rightarrow\mathsf{dgprop}$ to be the functor taking a dg-prop $P$ to the full subcategory of $\text{Tot}(Q(P))$ generated by the objects $\{(1)^n\}_{n\geq 0}$. To see that this defines a functor, recall from Construction \ref{const-Q-2} that for a morphism of dg-props $P\rightarrow P'$, the induced symmetric monoidal functor $\text{Tot}(Q(P))\rightarrow \text{Tot}(Q(P'))$ is the identity on object monoids, hence it restricts to a prop morphism $\tilde{P}\rightarrow \tilde{P}'$.

\bigskip \underline{The natural quasi-equivalence $\tilde{(-)}\rightarrow \id$.}

\bigskip Let $F|_{\tilde{P}}\colon \tilde{P}\rightarrow P$ be the composition
$$\tilde{P}\hookrightarrow \Tot(Q(P))\xrightarrow{F}P$$
It is clear that $F|_{\tilde{P}}$ induces an isomorphism on object monoids. By Lemma \ref{lem-Q-factorization}, $F|_{\tilde{P}}$ also induces quasi-isomorphisms on Hom-complexes, hence it is a quasi-equivalence. Naturality of $F$ and the inclusion $\tilde{P}\rightarrow \Tot(Q(P))$ imply that $F|_{\tilde{P}}$ is a natural quasi-equivalence.

\bigskip \underline{The natural transformation $\tilde{\alpha}$.}

\bigskip To produce the natural transformation $\tilde{\alpha}$, we use the transformation $\alpha$ from Lemma \ref{lem-alpha-TotQ}. Recall that there is an equivalence of categories
$$\Fun^{\otimes}(\Ass\otimes P,\Ch_k)\simeq \Fun^{\otimes}(P,\dgAlg_k).$$
The natural inclusion $i\colon \tilde{(-)}\rightarrow \text{Tot}(Q(-))$ gives us a natural transformation
$$i^*\colon \Fun^{\otimes}(\text{Tot}(Q(-)),\Ch_k)\rightarrow \Fun^{\otimes}(\tilde{(-)},\Ch_k)$$
and we define $\tilde{\alpha}$ to be the composition
$$\tilde{\alpha}=i^*\circ \alpha\colon  \Fun^{\otimes}(\Ass\otimes -,\Ch_k)\rightarrow \Fun^{\otimes}(\tilde{(-)},\Ch_k).$$
Because $\tilde{\alpha}$ is a restriction of $\alpha$, we have that for any prop $P$, and symmetric monoidal functor $\Phi\colon \Ass\otimes P\rightarrow \Ch_k$, there is an equality $\tilde{\alpha}_P(\Phi)(1)=\alpha_P(\Phi)(1)=C(\Phi(1))$, hence $\tilde{\alpha}$ has the stated properties.

\hfill $\square$

\begin{example} \label{example-chopf}
Consider the example $P=\CHopf$, the prop encoding a commutative Hopf algebra structure. Note that every morphism in $\CHopf$ is an algebra homomorphism, hence we have an equivalence $\Ass\otimes \CHopf\simeq \CHopf$ and Theorem A gives a recipe for the natural coherent commutative Hopf algebra structure on Hochschild chains of commutative Hopf algebras. In particular, $\tilde{\CHopf}$ is generated in degree 0 by the morphisms
$$()\xrightarrow{\eta}(1)$$
$$(1,1)\xrightarrow{\sh}(2)\xrightarrow{m}(1)$$
$$(1)\xrightarrow{\epsilon}()$$
$$(1)\xrightarrow{\Delta}(2)\xrightarrow{AW} (1,1)$$
$$(1)\xrightarrow{S}(1)$$
An example of a generator in degree 1 is the bialgebra relation, in which we need the homotopy $\theta\in \tilde{N}^{\Sigma}((2,2),(2,2))_1$ to interpolate between the upper and lower legs of the diagram. Here $F\in \Sigma_4$ is the transposition $(2,3)$.
\begin{center}
\begin{tikzpicture}
\node (A1) at (-4,0) {$(1,1)$};
\node (A2) at (0,0) {$(2,2)$};
\node (A3) at (4,0) {$(1,1,1,1)$};
\node (B1) at (-4,-3) {$(2)$};
\node (B2) at (0,-3) {$(4)$};
\node (B3) at (4,-3) {$(2,2)$};
\node (C1) at (-4,-6) {$(1)$};
\node (C2) at (0,-6) {$(2)$};
\node (C3) at (4,-6) {$(1,1)$};

\node (D1) at (1.75,-1.25) {$\theta$};
\node (D2) at (2,-1.5) {\rotatebox{225}{$\Longrightarrow$}};

\path[->,font=\scriptsize,>=angle 90]
(A1) edge node[auto] {$(\Delta,\Delta)$} (A2)
	 edge node[auto] {$\sh$} (B1)
(A2) edge node[auto] {$F\circ (AW,AW)$} (A3)
	 edge node[left] {$\Par_{(2,2)}(\sh)$} (B2)
(A3) edge node[auto] {$(\sh,\sh)$} (B3)
(B1) edge node[auto] {$\Par_{(\Delta,\Delta)}((2))$} (B2)
	 edge node[auto] {$m$} (C1)
(B2) edge node[below] {$\Par_{(2,2)}(AW)\circ F_*$} (B3)
	 edge node[left] {$\Par_{(m,m)}((2))\circ F_*$} (C2)
(B3) edge node[auto] {$(m,m)$} (C3)
(C1) edge node[auto] {$(\Delta)$} (C2)
(C2) edge node[auto] {$AW$} (C3);
\end{tikzpicture}
\end{center}
In a similar way, we need the contraction $\alpha_{(1,1)}\colon  AW\circ \sh \simeq \id\in \tilde{N}^{\Sigma}((1,1),(1,1))$ for the antipode diagrams.

Note that $\tilde{\CHopf}$ still has a strictly commutative multiplication. If $\mathcal{C}\mathbb{E}_n\mathcal{H}opf$ encodes commutative and $\mathbb{E}_n$-cocommutative Hopf algebras, then $\tilde{\mathcal{C}\mathbb{E}_n\mathcal{H}opf}$ will also be $\mathbb{E}_n$ cocommutative for $n\leq \infty$, but if $\mathcal{CCH}opf$ is the prop encoding a Hopf algebra structure which is both commutative and cocommutative, then $\tilde{\mathcal{CCH}opf}$ is strictly commutative but only $\mathbb{E}_{\infty}$-cocommutative, since $AW$ is not a symmetric monoidal transformation.
\end{example}


\begin{thebibliography}{9}

 

\bibitem{bv}
 Michael Boardman, Rainer Vogt,
 \emph{Homotopy invariant algebraic structures on topological spaces},
 Lecture Notes in Mathematics, Vol. 347, Springer-Verlag, (1973).

\bibitem{brun07}
 Morten Brun, Zbigniew Fiedorowicz, and Rainer M. Vogt,
 \emph{On the multiplicative structure of topological Hochschild homology},
 Algebr. Geom. Topol., 7:1633–-1650, 2007

\bibitem{dold61}
 Albrecht Dold,
 \emph{\"{U}ber die Steenrodschen Kohomologieoperationen},
 Ann. of Math. 73, 1961, 258--294.

\bibitem{em53}
 Samuel Eilenberg and Saunders Mac Lane,
 \emph{On the Groups $H(\Pi, n)$, I},
 Annals of Mathematics, Vol. 58, No. 1 (Jul., 1953), pp. 55--106.

\bibitem{em54}
 Samuel Eilenberg and Saunders MacLane,
 \emph{On the Groups $H(\Pi, n)$, II: Methods of Computation},
 Annals of Mathematics, Vol. 60, No. 1 (Jul., 1954), pp. 49--139.

\bibitem{fiorenza}
 D. Fiorenza,
 \emph{An introduction to the Language of Operads},
 \verb#www.mat.uniroma1.it/ fiorenza/ilo.ps.gz#,
 2006.

\bibitem{fresse08}
 Benoit Fresse,
 \emph{Props in model categories and homotopy invariance of structures},
 arXiv:0812.2738v4,
 5 Dec 2008.



\bibitem{loday92}
 Jean-Louis Loday,
 \emph{Cyclic Homology},
 Springer Verlag,
 1992.

\bibitem{maclanecatalg}
 Saunders MacLane,
 \emph{Categorical Algebra}
 Bull. Amer. Math. Soc., Volume 71, Number 1 (1965), 40--106.

\bibitem{markl}
 Martin Markl,
 \emph{Operads and PROPs},
 arXiv:math/0601129v3,
 6 Jan 2006.

\bibitem{richter00}
 Birgit Richter,
 \emph{$\mathbb{E}_\infty$-structure for $Q_*(R)$},
 Math. Ann. 316, 547--564 (2000).
  
\bibitem{richter03}
 Birgit Richter,
 \emph{Symmetry Properties of the Dold-Kan Correspondence},
 Mathematical Proceedings of the Cambridge Philosophical Society, 134(1), pp. 95-–102,
 2003.

\bibitem{ss03}
 Stefan Schwede and Brooke Shipley,
 \emph{Equivalences of monoidal model categories},
 Algebr. Geom. Topol., Volume 3, Number 1 (2003), 287--334.



\bibitem{wahl16}
 Nathalie Wahl and Craig Westerland,
 \emph{Hochschild homology of structured algebras},
 Advances in Math. 288 (2016), 240--307.

\bibitem{wahl12}
 Nathalie Wahl,
 \emph{Universal operations in Hochschild homology},
 J. Reine Angew. Math. 720 (2016), 81--127
 2016.

\bibitem{wolff73}
 Harvey Wolff,
 \emph{V-cat and V-graph},
 J. Pure Appl. Algebra 4 (1974), 123-–135.

\end{thebibliography}
\end{document}